\documentclass[11pt, reqno]{amsart}
\usepackage{amssymb, cite}
\usepackage{float}
\usepackage{mathtools}

\usepackage{mathrsfs}
\usepackage{color}

\mathtoolsset{showonlyrefs, showmanualtags}

\newtheorem{innercustomlem}{Lemma}
\newenvironment{customlem}[1]
  {\renewcommand\theinnercustomlem{#1}\innercustomlem}
  {\endinnercustomlem}

\makeatletter
\def\@setcopyright{}
\def\serieslogo@{}
\makeatother

\title[Slicing the stars]{Slicing the stars: counting algebraic numbers, integers, and units by degree and height\footnote{\today}\hfill}
\author[Robert Grizzard]{Robert Grizzard}
\address{Department of Mathematics, University of Wisconsin-Madison; 
480 Lincoln Drive, Madison, WI 53706 USA}
\email{grizzard@math.wisc.edu} 
\urladdr{www.math.wisc.edu/~grizzard}
\author[Joseph Gunther]{Joseph Gunther}
\address{Department of Mathematics, The Graduate Center, City University of New York (CUNY);
 365 Fifth Avenue, New York, NY 10016 USA}
\email{jgunther@gradcenter.cuny.edu}
\urladdr{sites.google.com/site/jgunther7}
\subjclass[2010]{11N45, 11G50, 11R06, 11P21, 11H16, 11R04}
\keywords{arithmetic statistics, height, Mahler measure, geometry of numbers}
\thanks{The second author was partially supported by National Science Foundation grant DMS-1301690.}
\numberwithin{equation}{section}

\begin{document} 
\baselineskip=14.5pt

\makeatletter
\def\imod#1{\allowbreak\mkern10mu({\operator@font mod}\,\,#1)}
\makeatother
\newcommand{\vectornorm}[1]{\left|\left|#1\right|\right|}
\newcommand{\ip}[2]{\left\langle#1,#2\right\rangle}
\newcommand{\ideal}[1]{\left\langle#1\right\rangle}
\newcommand{\sep}[0]{^{\textup{sep}}}
\newcommand{\Span}[0]{\operatorname{Span}}
\newcommand{\Tor}[0]{\operatorname{Tor}}
\newcommand{\Stab}[0]{\operatorname{Stab}}
\newcommand{\Orb}[0]{\operatorname{Orb}}
\newcommand{\Soc}[0]{\operatorname{Soc}}
\newcommand{\Kernel}[0]{\operatorname{ker }}
\newcommand{\Gal}[0]{\operatorname{Gal}}
\newcommand{\Aut}[0]{\operatorname{Aut}}
\newcommand{\Out}[0]{\operatorname{Out}}
\newcommand{\Inn}[0]{\operatorname{Inn}}
\newcommand{\Image}[0]{\operatorname{im }}
\newcommand{\tr}[0]{\operatorname{tr }}
\newcommand{\Tr}[0]{\operatorname{Tr }}
\newcommand{\Nm}[0]{\operatorname{Nm }}
\newcommand{\vol}[0]{\operatorname{vol }}
\newcommand{\pgl}[0]{\operatorname{PGL}_2(\mathbb{F}_3)}
\newcommand{\AGL}[0]{\operatorname{AGL}}
\newcommand{\PGL}[0]{\operatorname{PGL}}
\newcommand{\PSL}[0]{\operatorname{PSL}}
\newcommand{\PSp}[0]{\operatorname{PSp}}
\newcommand{\PSU}[0]{\operatorname{PSU}}
\newcommand{\Sp}[0]{\operatorname{Sp}}
\newcommand{\POmega}[0]{\operatorname{P\Omega}}
\newcommand{\Uup}[0]{\operatorname{U}}
\newcommand{\Gup}[0]{\operatorname{G}}
\newcommand{\mcFup}[0]{\operatorname{F}}
\newcommand{\Eup}[0]{\operatorname{E}}
\newcommand{\Bup}[0]{\operatorname{B}}
\newcommand{\Dup}[0]{\operatorname{D}}
\newcommand{\Mup}[0]{\operatorname{M}}
\newcommand{\mini}[0]{\operatorname{min}}
\newcommand{\maxi}[0]{\textup{max}}
\newcommand{\modu}[0]{\textup{mod}}
\newcommand{\nth}[0]{^\textup{th}}
\newcommand{\sd}[0]{\leq_{sd}}
\newcommand{\ld}[0]{\leq \hspace{-3pt}}
\makeatletter
\newcommand*{\house}[1]{%
  \mathord{%
    \mathpalette\@house{#1}%
  }%
}
\newcommand*{\@house}[2]{%
  % #1: math style
  % #2: expression that gets the "house"
  % get the line width of `\overline' in the current math font size
  \dimen@=\fontdimen8 %
      \ifx#1\scriptscriptstyle\scriptscriptfont
      \else\ifx#1\scriptstyle\scriptfont
      \else\textfont\fi\fi
      3 %
  \sbox0{%
    $#1%
      \vrule width\dimen@\relax
      \overline{%
        \kern2\dimen@
        \begingroup % to keep changes of \dimen@ in #2 local
          #2%
        \endgroup
        \kern2\dimen@
      }%
      \vrule width\dimen@\relax
      \mathsurround=1.5\dimen@ % outside margin
    $%
  }%
  % TeX adds an empty space above `\overline', it needs to be
  % removed to get the correct height for the `\vrule's
  \ht0=\dimexpr\ht0-\dimen@\relax
  \dp0=\dimexpr\dp0+2\dimen@\relax
  \vbox{%
    \kern\dimen@ % reinsert previously removed space
    \copy0 %
  }%
}
 \newtheorem{theorem}{Theorem}[section]
 \renewcommand{\thetheorem}{\arabic{section}.\arabic{theorem}}
 \newtheorem{proposition}[theorem]{Proposition}
   \renewcommand{\theproposition}{\arabic{section}.\arabic{theorem}}
 \newtheorem{lemma}[theorem]{Lemma}
  \renewcommand{\thelemma}{\arabic{section}.\arabic{theorem}}
 \newtheorem{corollary}[theorem]{Corollary}
   \renewcommand{\thecorollary}{\arabic{section}.\arabic{theorem}}
 \newtheorem{conjecture}[theorem]{Conjecture}
   \renewcommand{\theconjecture}{\arabic{section}.\arabic{theorem}}
 \newtheorem{definition}[theorem]{Definition}
   \renewcommand{\thedefinition}{\arabic{section}.\arabic{theorem}}
 \newtheorem{remark}[theorem]{Remark}
   \renewcommand{\thedefinition}{\arabic{section}.\arabic{theorem}}
 \newtheorem{question}[theorem]{Question}
   \renewcommand{\thequestion}{\arabic{section}.\arabic{theorem}}
 \newtheorem*{claim*}{Claim}
 \newtheorem{claim}[theorem]{Claim}
   \renewcommand{\theclaim}{\arabic{section}.\arabic{theorem}}
 \newcommand{\mc}{\mathcal}
 \newcommand{\mf}{\mathfrak}
 \newcommand{\mcA}{\mc{A}}
 \newcommand{\mcB}{\mc{B}}
 \newcommand{\mcC}{\mc{C}}
 \newcommand{\mcD}{\mc{D}}
 \newcommand{\mcE}{\mc{E}}
 \newcommand{\mcF}{\mc{F}}
 \newcommand{\mcG}{\mc{G}}
 \newcommand{\mcH}{\mc{H}}
 \newcommand{\mcI}{\mc{I}}
 \newcommand{\mcJ}{\mc{J}}
 \newcommand{\mcK}{\mc{K}}
 \newcommand{\mcL}{\mc{L}}
 \newcommand{\mcM}{\mc{M}}
 \newcommand{\mcN}{\mc{N}}
 \newcommand{\mcO}{\mc{O}}
 \newcommand{\mcP}{\mc{P}}
 \newcommand{\mcQ}{\mc{Q}}
 \newcommand{\mcR}{\mc{R}}
 \newcommand{\mcS}{\mc{S}} 
 \newcommand{\mcT}{\mc{T}}
 \newcommand{\mcU}{\mc{U}}
 \newcommand{\mcV}{\mc{V}}
 \newcommand{\mcW}{\mc{W}}
 \newcommand{\mcX}{\mc{X}}
 \newcommand{\mcY}{\mc{Y}}
 \newcommand{\mcZ}{\mc{Z}}
 \newcommand{\mfp}{\mf{p}}
 \newcommand{\mfP}{\mf{P}}
 \newcommand{\mfq}{\mf{q}}
 \newcommand{\mfQ}{\mf{Q}}  
 \newcommand{\va}{\vec a}
 \newcommand{\vb}{\vec b}
 \newcommand{\vn}{\vec n}
\newcommand{\vx}{\vec x}
\newcommand{\vy}{\vec y}
\newcommand{\vu}{\vec u}
\newcommand{\vv}{\vec v}
\newcommand{\vl}{\vec\ell}
\newcommand{\vr}{\vec r}
\newcommand{\vs}{\vec s}
\newcommand{\vw}{\vec w}
 \newcommand{\oQ}{\overline{\QQ}}
 \newcommand{\AAA}{\mathbb{A}}
 \newcommand{\CC}{\mathbb{C}}
 \newcommand{\FF}{\mathbb{F}}
  \newcommand{\GG}{\mathbb{G}}
 \newcommand{\NN}{\mathbb{N}}
 \newcommand{\PP}{\mathbb{P}}
 \newcommand{\QQ}{\mathbb{Q}}
 \newcommand{\QQd}{\QQ^{(d)}}
 \newcommand{\RR}{\mathbb{R}}
 \newcommand{\ZZ}{\mathbb{Z}}
 \newcommand{\oneto}[1]{\{1,\dots,#1\}}
 \newcommand{\tors}[0]{\textup{tors}}
 \newcommand{\QQbar}[0]{\overline{\QQ}}
 \newcommand{\vep}{\varepsilon}
 %Vaaler macros
 \newcommand{\A}[0]{\mcA}
 \newcommand{\B}[0]{\mcB}
 \newcommand{\D}[0]{\mcD}
 \newcommand{\E}[0]{\mcE}
 \newcommand{\F}[0]{\mcF}
 \newcommand{\G}[0]{\mcG} 
 \newcommand{\eL}[0]{\mcL} 
 \newcommand{\M}[0]{\mcM} 
 \newcommand{\X}[0]{\mcX}
 \newcommand{\Q}[0]{\QQ}
 \newcommand{\Z}[0]{\ZZ}
 \newcommand{\zZ}{\mcZ}
 \newcommand{\R}[0]{\RR}
 \newcommand{\spann}[0]{\Span}
 \newcommand{\closure}[0]{\operatorname{closure}}
 \newcommand{\aA}{\mathfrak A}
 \newcommand{\bB}{\mathfrak B}
 \newcommand{\ff}{\mathfrak F}
 \newcommand{\fp}{\mathfrak p}
 \newcommand{\fb}{f_{\beta}}
 \newcommand{\fg}{f_{\gamma}}
 \newcommand{\gb}{g_{\beta}}
 \newcommand{\ep}{\varepsilon}
 \newcommand{\vphi}{\varphi}
 \newcommand{\bo}{\boldsymbol 0}
 \newcommand{\ba}{\boldsymbol a}
 \newcommand{\bb}{\boldsymbol b}
 \newcommand{\bm}{\boldsymbol m}
 \newcommand{\bgamma}{\boldsymbol \gamma}
 \newcommand{\bt}{\boldsymbol t}
 \newcommand{\bu}{\boldsymbol u}
 \newcommand{\bv}{\boldsymbol v}
 \newcommand{\bx}{\boldsymbol x}
 \newcommand{\bwy}{\boldsymbol y}
 \newcommand{\bxi}{\boldsymbol \xi}
 \newcommand{\bbeta}{\boldsymbol \eta}
 \newcommand{\bw}{\boldsymbol w}
 \newcommand{\bz}{\boldsymbol z}
 \newcommand{\whG}{\widehat{G}}
 \newcommand{\oK}{\overline{K}}
 \newcommand{\oKt}{\overline{K}^{\times}}
 \newcommand{\oq}{\oQ^{\times}}
 \newcommand{\oQt}{\oQ^{\times}/\Tor\bigl(\oQ^{\times}\bigr)}
 \newcommand{\ot}{\Tor\bigl(\oQ^{\times}\bigr)}
 \newcommand{\h}{\frac12}
 \newcommand{\hh}{\tfrac12}
 \newcommand{\dx}{\text{\rm d}x}
 \newcommand{\dy}{\text{\rm d}y}
 \newcommand{\dmu}{\text{\rm d}\mu}
 \newcommand{\dnu}{\text{\rm d}\nu}
 \newcommand{\dla}{\text{\rm d}\lambda}
 \newcommand{\dlav}{\text{\rm d}\lambda_v}
 \newcommand{\trho}{\widetilde{\rho}}
 \newcommand{\dtrho}{\text{\rm d}\widetilde{\rho}}
 \newcommand{\drho}{\text{\rm d}\rho}
 \newcommand{\wS}{\widetilde{S}}
 \newcommand{\wT}{\widetilde{T}}
   \def\bbbone{{\mathchoice {\rm 1\mskip-4mu l} {\rm 1\mskip-4mu l}
   {\rm 1\mskip-4.5mu l} {\rm 1\mskip-5mu l}}}

 \renewcommand{\thefootnote}{\fnsymbol{footnote}}
 % 1   2    3     4     5     6     7     8     9 
 % * \dag                          **

% \def\today{%\number\time, 
%   \ifcase\month\or
%   January\or February\or March\or April\or May\or June\or
%   July\or August\or September\or October\or November\or December\fi
%   \space\number\day, \number\year}
\bibliographystyle{alpha} 

%%%%%%%%%%%%%%%%%%%%%%%%%%%%%%%%%%%%%%%%%%%%%%%%%%%%%%%%%%%%%%
\begin{abstract} 
Masser and Vaaler have given an asymptotic formula for the number of algebraic numbers of given degree $d$ and increasing height.  This problem was solved by counting lattice points (which correspond to minimal polynomials over $\ZZ$) in a homogeneously expanding star body in $\RR^{d+1}$.  The volume of this star body was computed by Chern and Vaaler, who also computed the volume of the codimension-one ``slice'' corresponding to monic polynomials -- this led to results of Barroero on counting algebraic integers.  We show how to estimate the volume of higher-codimension slices, which allows us to count units, algebraic integers of given norm, trace, norm and trace, and more.  We also refine the lattice point-counting arguments of Chern-Vaaler to obtain explicit error terms with better power savings, which lead to explicit versions of some results of Masser-Vaaler and Barroero.
\end{abstract}

%%%%%%%%%%%%%%%%%%%%%%%%%%%%%%%%%%%%%%%%%%%%%%%%%%%%%%%%%%%%%%%
\maketitle  

\setcounter{tocdepth}{1}
\tableofcontents

%1
\section{Introduction}\label{introductionsec}
 A classical theorem of Northcott %(see for example \cite[Theorem 1.6.8]{bombierigubler}) 
states that there are only finitely elements of $\QQbar$ of bounded degree and height.  It's then natural to ask, for interesting subsets $\mcS \subset \QQbar$ of bounded degree, how the number of elements of bounded height grows as we let the height bound increase.  More precisely, one considers the asymptotics of 
\begin{equation}
N(\mcS,\mcH)=\#\{x \in \mcS \ | \ H(x) \leq \mcH\},
\end{equation}
where $H(x)$ is the absolute multiplicative Weil height of $x$ (see for example \cite[p. 16]{bombierigubler}).

Many of the oldest instances of such asymptotic statements concern elements of a fixed number field.  Schanuel \cite[Corollary]{schanuel79} proved that, for any number field $K$, as $\mcH$ grows, 
\begin{equation}
N(K,\mcH) = c_K\cdot \mcH^{2[K:\QQ]} + O\left(\mcH^{2[K:\QQ]-1}\operatorname{log }\mcH\right),
\end{equation}
where the constant $c_K$ involves all the classical invariants of the number field $K$, and the $\log \mcH$ factor disappears for $K \neq \QQ$.

Lang states analogous asymptotics for the ring of integers $\mcO_K$ and its unit group $\mcO_K^*$ \cite[Chapter 3, Theorem 5.2]{lang}:
\begin{align}
N(\mcO_K,\mcH) &= \gamma_K\cdot \mcH^{[K:\QQ]}(\operatorname{log }\mcH)^r + O\left(\mcH^{[K:\QQ]}(\operatorname{log }\mcH)^{r-1}\right);\\
N(\mcO_K^*,\mcH) &= \gamma_K^*\cdot (\operatorname{log }\mcH)^{r} + O\left((\operatorname{log }\mcH)^{r-1}\right),
\end{align}
where $r$ is the rank of $\mcO_K^*$ and $\gamma_K$ and $\gamma_K^*$ are unspecified constants.  That first count was later refined to a multi-term asymptotic by Widmer \cite[Theorem 1.1]{Widmer15}.  

More recently, natural subsets that aren't contained within a single number field have been examined.  Masser and Vaaler \cite[Theorem]{masservaaler1} determined the asymptotic for the entire set $\QQbar_d = \{x \in \QQbar \ | \ [\QQ(x):\QQ] = d\}$:
\begin{equation}\label{mvthm}
N(\QQbar_d,\mcH) = \frac{d\cdot V_d}{2\zeta(d+1)}\cdot \mcH^{d(d+1)} + O\left(\mcH^{d^2}(\operatorname{log }\mcH)\right),
\end{equation} 
where the $\log \mcH$ factor disappears for $d \geq 3$, and $V_d$ is an explicit positive constant that we'll define shortly.

This asymptotic was deduced from results of Chern and Vaaler \cite{chernvaaler} (discussed at length in section \ref{starsec}), which also imply an asymptotic for the set $\mcO_d$ of all algebraic integers of degree $d$, as noted in Widmer \cite[(1.2)]{Widmer15}.  It was sharpened by Barroero \cite[Theorem 1.1, case $k=\QQ$]{barroero14}:

\begin{equation}\label{barthm}
N(\mcO_d,\mcH) = d\cdot V_{d-1}\cdot \mcH^{d^2} + O\left(\mcH^{d(d-1)}(\operatorname{log }\mcH)\right),
\end{equation}
where again the $\log \mcH$ factor disappears for $d \geq 3$.

After algebraic numbers and integers, it's natural to turn to the problem of counting units and other interesting sets of algebraic numbers.  It's also desirable to obtain versions of these estimates with explicit error terms.  These are the two purposes of this paper.  

We establish counts of units, algebraic numbers of given norm, given trace, and given norm and trace in Corollaries \ref{unitcor}-\ref{normtracecor}, which follow from the more general Theorem \ref{maincor} stated below.  As for explicit error bounds, we have made several improvements to the existing literature.  The lack of explicit error terms in the results (\refeq{mvthm}) and (\refeq{barthm}) is inherited from results of Chern and Vaaler on counting polynomials.  Specifically, Chern and Vaaler mention (see \cite[p. 6]{chernvaaler}) that it would be of interest to make the implied constant in \cite[Theorem 3]{chernvaaler} explicit, but they were unable to do so.  In this paper we are able to make this constant explicit (Theorem \ref{genpolycount} below), and we also prove an analogous result for monic polynomials (Theorem \ref{moniccount}).  We use these to obtain versions of (\refeq{mvthm}) and (\refeq{barthm}) that are uniform in both $\mcH$ and $d$.  These, along with an explicit version of our result on counting units, are summarized below in Theorem \ref{exsum}.

\subsection{Results}
%In this paper, we generalize (\refeq{mvthm}) and (\refeq{barthm}), and bound the error terms explicitly.  
%To explain our generalization, we need one definition.  
Throughout the paper, we will understand the {\em minimal polynomial} of an algebraic number to be its minimal polynomial over $\ZZ$; we obtain this by multiplying the traditional minimal polynomial over $\QQ$ by the smallest positive integer such that all its coefficients become integers.

Counting algebraic integers, as in (\refeq{barthm}), is equivalent to counting only those algebraic numbers whose minimal polynomial has leading coefficient 1.  Our primary goal in this paper is to count algebraic numbers of fixed degree and bounded height subject to specifying {\em any} number of the leftmost and rightmost coefficients of their minimal polynomials.  Besides specializing to the cases of algebraic numbers and algebraic integers above, this will allow us to count units, algebraic integers with given norm, algebraic integers with given trace, and algebraic integers with given norm and trace.

To state our theorem, we need a little notation.  Our asymptotic counts will involve the Chern-Vaaler constants
\begin{equation}\label{vddef}
V_d = 2^{d+1}(d+1)^s \prod_{j=1}^s \frac{(2j)^{d-2j}}{(2j+1)^{d+1-2j}},
\end{equation}
where $s = \lfloor(d-1)/2\rfloor.$  These constants are volumes of certain star bodies discussed later.

For integers $m$, $n$, and $d$ with $0 < m$, $0 \leq n$, and $m+n \leq d$, and integer vectors $\vl \in \ZZ^m$ and $\vr \in \ZZ^n$, we write $\mcN(d,\vl,\vr,\mcH)$ for the number of algebraic numbers of degree $d$ and height at most $\mcH$, whose minimal polynomial is of the form \begin{equation}
f(z) = \ell_0 z^d + \cdots + \ell_{m-1}z^{d-(m-1)}+x_mz^{d-m} + \cdots + x_{d-n}z^n + r_{d-n+1}z^{n-1} + \cdots + r_d.
\end{equation}
Lastly, we set $g = d-m-n$.  In the statements below, the implied constants depend on all parameters stated other than $\mcH$.
\begin{theorem}\label{maincor}
Fix $d$, $\vl \in \ZZ^m$, and $\vr \in \ZZ^n$ as above.  Assume that $\ell_0 >0$, that
\begin{equation}
\gcd(\ell_0,\dots,\ell_{m-1},r_{d-n+1},\dots,r_d) = 1,
\end{equation}
and that $r_d \neq 0$ if $n>0$.  Then as $\mcH \to \infty$ we have
\begin{equation}
\mcN(d,\vl,\vr,\mcH) = d\cdot V_g\cdot \mcH^{d(g+1)} + O\left(\mcH^{d(g+\frac{1}{2})}\log \mcH\right).
\end{equation}
\end{theorem}

This generalizes the situation one faces when counting algebraic integers, whose minimal polynomials are monic ($m=1$, $n=0$, $\vl = (1)$).  Certain special cases are of particular interest, and we prove stronger power savings terms for them.

\begin{corollary}\label{unitcor}
Let $d \geq 2$, and let $N(\mcO^*_d,\mcH)$ denote the number of units in the algebraic integers of height at most $\mcH$ and degree $d$ over $\QQ$.  Then as $\mcH \to \infty$ we have
\begin{equation}
N(\mcO^*_d,\mcH) = 2d\cdot V_{d-2}\cdot \mcH^{d(d-1)} + O\left(\mcH^{d(d-2)}\right).
\end{equation}
%[$d=2$ log factors will be discussed]
\end{corollary}

\begin{corollary}\label{normcor}
Let $\nu \neq 0$ be an integer, $d \geq 2$, and let $\mcN_{\Nm=\nu}(d,\mcH)$ denote the number of algebraic integers with norm $\nu$, of height at most $\mcH$ and degree $d$ over $\QQ$.  Then as $\mcH \to \infty$ we have
\begin{equation}
\mcN_{\Nm=\nu}(d,\mcH) = d\cdot V_{d-2}\cdot \mcH^{d(d-1)} + O\left(\mcH^{d(d-2)}\right).
\end{equation}
%where the implied constant depends only on $d$ and $\nu$.
%[$d=2$??]
\end{corollary}

\begin{corollary}\label{tracecor}
Let $\tau$ be an integer, $d \geq 2$, and let $\mcN_{\Tr=\tau}(d,\mcH)$ denote the number of algebraic integers with trace $\tau$, of height at most $\mcH$ and degree $d$ over $\QQ$.  Then as $\mcH \to \infty$ we have
%\begin{align}
%\mcN_{\Tr=\tau}(2,\mcH) &= 2\cdot V_{0}\cdot \mcH^{2} + O\left(\mcH \right),\\
%\mcN_{\Tr=\tau}(3,\mcH) &= 3\cdot V_{1}\cdot \mcH^{6} + O\left(\mcH^{3} \log \mcH \right),~\textup{and}\\
%\mcN_{\Tr=\tau}(d,\mcH) &= d\cdot V_{d-2}\cdot \mcH^{d(d-1)} + O\left(\mcH^{d(d-2)}\right), ~\textup{for}~ d \geq 4.
%\end{align}
\begin{equation}
\mcN_{\Tr=\tau}(d,\mcH) = d\cdot V_{d-2}\cdot \mcH^{d(d-1)} + 
\left\{
\begin{array}{ll}
O\left(\mcH \right), &\textup{if}~ d =2\vspace{7pt}\\
O\left(\mcH^{3} \log \mcH \right), &\textup{if}~ d =3\vspace{7pt}\\%~\textup{and}\\
O\left(\mcH^{d(d-2)}\right), &\textup{if}~ d \geq 4.
\end{array}
\right.
\end{equation}
%where the implied constant depends only on $d$ and $\tau$.
%[$d=2,3$ log factors will be discussed]
\end{corollary}

\begin{corollary}\label{normtracecor}
Let $\nu \neq 0$ and $\tau$ be integers, $d \geq 3$, and let $\mcN_{\Nm=\nu,\Tr = \tau}(d,\mcH)$ denote the number of algebraic integers with norm $\nu$, trace $\tau$, of height at most $\mcH$ and degree $d$ over $\QQ$.  Then as $\mcH \to \infty$ we have
\begin{equation}
\mcN_{\Nm=\nu,\Tr=\tau}(d,\mcH) = d\cdot V_{d-3}\cdot \mcH^{d(d-2)} + O(\mcH^{d(d-3)}).
\end{equation}
%where the implied constant depends only on $d$, $\nu$, and $\tau$.
%[$d=3$ log factors will be discussed]
\end{corollary}

\begin{remark}\normalfont
In Corollaries \ref{normcor} through \ref{normtracecor}, the main term of the asymptotic doesn't depend on the specific coefficients being enforced.  Thus %in some sense 
these may be interpreted as results on the equidistribution of norms and traces.\end{remark}
\begin{remark}\label{maninremark}\normalfont
%Version 1: Counting points of degree $d$ and bounded height in $\QQbar$, or equivalently on $\mathbb{P}^1$, corresponds to counting rational points of bounded height on its $d$-th symmetric product $\mathbb{P}^d$; this is implicit in our method, which counts minimal polynomials.  Thus the main terms of our asymptotic counts should be related to integral point versions of Manin's conjecture CITATION OR TWO, as suggested in LE RUDULIER REMARK in relation to Barroero's results.  However, it seems unlikely that those general results would give error terms as strong as ours.

The type of counts found in this paper are related to Manin's conjecture, which addresses the asymptotic number of rational points of bounded height on Fano varieties.  Counting points of degree $d$ and bounded height in $\QQbar$, or equivalently, on $\mathbb{P}^1$, can be transferred to a question of counting rational points of bounded height on the $d$-th symmetric product of $\mathbb{P}^1$, which is $\mathbb{P}^d$.  This is what Masser and Vaaler implicitly do when they count algebraic numbers by counting their minimal polynomials (as does this paper; see the Methods subsection below).  However, one needs to use a non-standard height on $\mathbb{P}^d$; Le Rudulier takes this approach explicitly \cite[Th\'eor\`eme 1.1]{lerudulier}, thereby re-proving and generalizing (the main term of) the result of Masser and Vaaler.  It should be noted, though, that while the shape of the main term -- a constant times the appropriate power of the height -- follows from known results on Manin's conjecture, {\em explicitly} determining the constant in front relies ultimately on an archimedean volume calculation of Chern and Vaaler.

Barroero's count of algebraic integers of degree $d$ corresponds to counting rational points on $\mathbb{P}^d$ that are integral with respect to the hyperplane at infinity.  As noted in \cite[Remarque 5.3]{lerudulier}, the shape of the count's main term then follows from general results of Chambert-Loir and Tschinkel on counting integral points of bounded height on equivariant compactifications of affine spaces \cite[Theorem 3.5.6]{clt}.

Our own units count corresponds to counting points on $\mathbb{P}^d$ integral with respect to {\em two} hyperplanes.  Again, the shape of the main term -- a constant times the correct power of the height -- follows from general integral point counts for toric varieties \cite[Theorem 3.11.5]{thereisnoorderinwhichyoucanwritetoricandcltsuchthatthereisntaconfusingdoubleletter}.  However, that constant is expressed as a product of local integrals and Galois-cohomological invariants.  It is unclear to the authors of this paper whether the constant can be calculated explicitly without knowledge of the volumes of slices we compute.  Regardless, the error terms obtained by using the general toric results are significantly weaker than those in this paper, and their dependence on $d$ cannot be made explicit.
\end{remark}
The second goal of this paper is to give explicit error terms, which we feel is especially justified in this context, beyond general principles of error-term morality.  Namely, it's natural to ask questions about properties of ``random algebraic numbers" (or random algebraic integers, random units, etc.).  For example: ``What's the probability that a random element of $\QQbar$ generates a Galois extension of $\QQ$?"

How to make sense of a question like this?  There are models from other arithmetic contexts; for example, if we're asked ``What's the probability that a random positive integer is square-free?" we know what to do: count the number of square-free integers from $1$ to $N$, divide that by $N$, and ask if that proportion has a limit as $N$ grows (Answer: Yes, $\frac{6}{\pi^2}$).  Note that the easiest part is dividing by $N$, the number of elements in your finite box.  In order to make sense of probabilistic statements in the context of $\QQbar$, one would like to first take a box of bounded height and degree (which will have only finitely many algebraic numbers by Northcott), determine the relevant proportion within that finite box, and then let the box size grow.   But now the denominator in question is far from trivial; unlike counting the number of integers from $1$ to $N$, estimating how many algebraic numbers are in a height-degree box is a more delicate matter.  %In this, it's somewhat more akin to counting numbers fields of bounded discriminant.

In the context of $\QQbar$, where there are {\em two} natural parameters to increase (the height and the degree), the gold standard for a ``probabilistic" result would be that it holds for any increasing set of height-degree boxes such that the minimum of the height and degree goes to infinity.  To prove results that even approach this standard (e.g. one might require that the height of the boxes grows at least as fact as some function of the degree), one likely needs good estimates for how many numbers are in a height-degree box to begin with.  Without an estimate that holds uniformly in both $\mcH$ and $d$, one would be justified in making statements about random elements in $\QQbar$ of fixed degree $d$, but not random elements of $\QQbar$ overall. Thus controlling the error terms in the theorems above is crucial.

\begin{figure}[H]
\begin{center}\includegraphics[width=\textwidth, keepaspectratio=true]{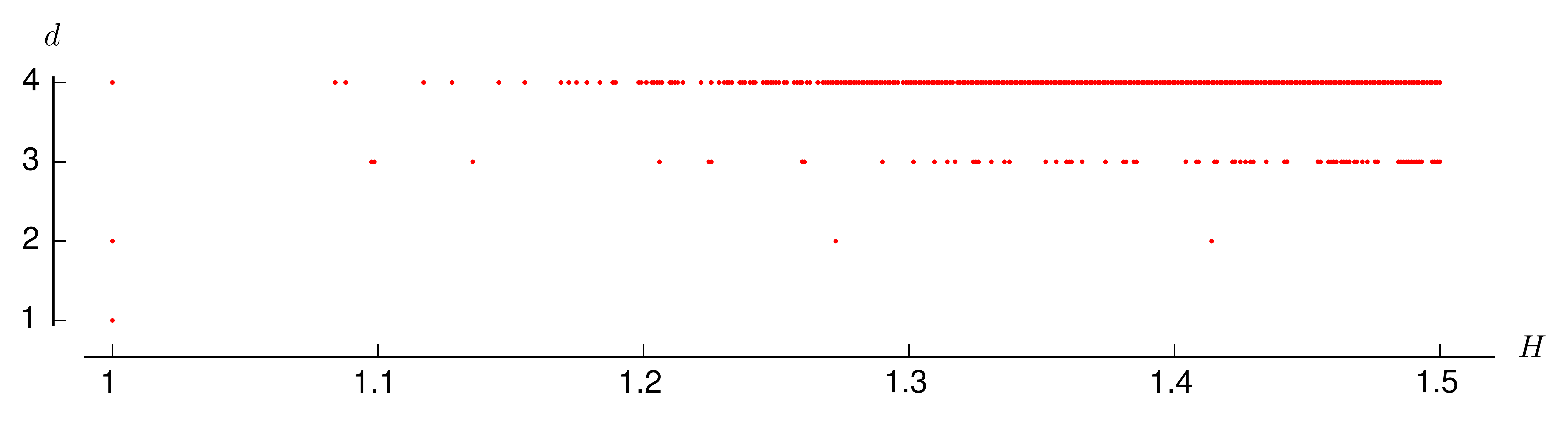}\end{center}
\caption{Algebraic numbers of degree $d \leq 4$ and height $H \leq 1.5$.  Each dot represents $d$ conjugate algebraic numbers.}
\end{figure} 

To this end, in this paper we give explicit error bounds for the algebraic number counts of Masser and Vaaler, the algebraic integer counts of Barroero, and our own unit counts.  Below $p_d(T)$ is a polynomial defined in Section \ref{starsec} whose leading term is $V_{d-1} T^d$, so our result is consistent with (\refeq{barthm}). 

%\begin{theorem}\label{exsum}
%For all $d \geq 3$ we have
%\begin{enumerate}
%\item[(i)] $\displaystyle{
% \left| N(\QQbar_d,\mcH) - \frac{d\cdot V_d}{2\zeta(d+1)}\mcH^{d(d+1)} \right| \leq 3.37 \cdot (15.01)^{d^2}\cdot \mcH^{d^2}, \hspace{8pt}\textup{for}~\mcH \geq 1}$;\vspace{7pt}
%\item[(ii)] $\displaystyle{
% \left|N(\mcO_d,\mcH) - d p_d(\mcH^{d}) \right| \hspace{35.6pt}\leq 1.13 \cdot 4^d d^{d} 2^{d^2}\cdot \mcH^{d(d-1)}, \hspace{1pt}\textup{for}~\mcH \geq 1}$; and\vspace{7pt}
%\item[(iii)] $\displaystyle{
%\left| N(\mcO^*_d,\mcH) - 2d V_{d-2}\cdot \mcH^{d(d-1)}\right| \hspace{2.5pt}\leq 0.0000126\cdot d^3 4^d (15.01)^{d^2} \cdot \mcH^{d(d-1)-1},} \\
%\phantom{.}\hspace{250pt}\textup{for}~ \mcH \geq d2^{d+1/d}
%$.
%\end{enumerate}
%\end{theorem}
\begin{theorem}\label{exsum}
Let $\QQbar_d$ denote the set of algebraic numbers of degree $d$ over $\QQ$, let $\mcO_d$ denote the set of algebraic integers of degree $d$ over $\QQ$, and let $\mcO^*_d$ denote the set of units of degree $d$ over $\QQ$ in the ring of all algebraic integers.  For all $d \geq 3$ we have
\begin{equation}
\begin{array}{lll}
\textup{(I\phantom{ii})}~
 \left| N(\QQbar_d,\mcH) - \frac{d\cdot V_d}{2\zeta(d+1)}\mcH^{d(d+1)} \right| &\leq 3.37 \cdot (15.01)^{d^2}\cdot \mcH^{d^2}, &\textup{for}~\mcH \geq 1;\vspace{7pt}\\
\textup{(ii\phantom{i})}~
 \left|N(\mcO_d,\mcH) - d p_d(\mcH^{d}) \right| &\leq 1.13 \cdot 4^d d^{d} 2^{d^2}\cdot \mcH^{d(d-1)}, &\textup{for}~\mcH \geq 1; and\vspace{7pt}\\
\textup{(iii)}~
\left| N(\mcO^*_d,\mcH) - 2d V_{d-2}\cdot \mcH^{d(d-1)}\right| &\leq 0.0000126\cdot d^3 4^d (15.01)^{d^2} &\hspace{-10pt} \cdot ~\mcH^{d(d-1)-1}, \vspace{7pt}\\
&&\textup{for}~ \mcH \geq d2^{d+1/d}
.
\end{array}
\end{equation}
\end{theorem}

\subsection{Methods}
The starting point of all our proofs is the relationship between the height of an algebraic number and the Mahler measure of its minimal polynomial.  Recall that the Mahler measure $\mu(f)$ of a polynomial with complex coefficients
\begin{equation}
f(z) = w_0z^d + w_1z^{d-1} +\cdots + w_d = w_0(z-\alpha_1)\cdots(z-\alpha_d) \in \CC[z],
\end{equation}
with $w_0 \not = 0$, is defined by
\begin{equation}
\mu(f) = |w_0| \prod_{i=1}^d \max\{1,|\alpha_i|\},
\end{equation}
and $\mu(0)$ is defined to be zero.  It's immediate that the Mahler measure is multiplicative: $\mu(f_1f_2)=\mu(f_1)\mu(f_2)$.

Crucially for our purposes, if $f(z)$ is the minimal polynomial of an algebraic number $\alpha$, then we have (see for example \cite[Proposition 1.6.6]{bombierigubler})
\begin{equation}
\mu(f) = H(\alpha)^d.
\end{equation}Thus, in order to count degree $d$ algebraic numbers of height at most $\mcH$, we can instead count integer polynomials of Mahler measure at most $\mcH^d$.

We identify a polynomial with its vector of coefficients, so that counting integer polynomials amounts to counting lattice points.  To do this we employ techniques from the geometry of numbers, which make rigorous the idea that, for a reasonable subset of Euclidean space, the number of integer lattice points in the set should be approximated by its volume.  So for example, the number of integer polynomials with degree at most $d$ and Mahler measure at most $T$ should be roughly the volume of the set of such \emph{real} polynomials $$\{f \in \RR[z]_{\operatorname{deg }\leq d} ~\big|~ \mu(f) \leq T\} \subset \RR^{d+1}.$$
Note that by multiplicativity of the Mahler measure, this set is the same as $T\mcU_d$, where 
\begin{equation}
\mcU_d := \{f \in \RR[z]_{\operatorname{deg }\leq d} ~\big|~ \mu(f) \leq 1\}.
\end{equation}

The set $\mcU_d$ will be our primary object of study.  It is a closed, compact ``star body,'' i.e. a subset of euclidean space closed under scaling by numbers in $[0,1]$.
Chern and Vaaler \cite[Corollary 2]{chernvaaler} explicitly determined the volume of $\mcU_d$.  In a rather heroic calculation, they showed that $V_d := \vol_{d+1}(\mcU_d)$ is given by the positive rational number in (\refeq{vddef})\footnote{Our $\mcU_d$ is the same as what would be denoted by $\mathscr{S}_{d+1}$ in the notation of \cite{chernvaaler}, and our $V_d$ matches their $V_{d+1}$.  Our subscripts correspond to the degree of the polynomials being counted rather than the dimension of the space.}.  Thus by geometry of numbers, and noting that $\vol(T\mcU_d) = T^{d+1} \cdot \vol(\mcU_d),$ one expects the number of integer polynomials of degree at most $d$ and Mahler measure at most $T$ to be approximately $T^{d+1} \cdot V_d$.  Chern and Vaaler proved this is indeed the case.  Masser and Vaaler then showed how to refine this count of all such polynomials to just minimal polynomials, which let them prove the algebraic number count in (\refeq{mvthm}).

What if you only want to count algebraic integers?  Again, the above approach suggests you should do that by counting their minimal polynomials.  Algebraic integers are characterized by having \emph{monic} minimal polynomials.  Thus one is naturally led to seek the volume of the ``monic slice" of $T \mcU_d$ consisting of those real polynomials with leading coefficient 1. However, these slices are no longer dilations of each other, so their volumes aren't determined by knowing the volume of one such slice.  Still, Chern and Vaaler were able to compute the volumes of monic slices of $T\mcU_d$; rather than a constant times a power of $T$, they are given by a polynomial in $T$, whose leading term is $V_{d-1} T^d$.  Geometry of numbers can then be applied again to obtain the algebraic integer count in (\refeq{barthm}).

In order to count units of degree $d$, or algebraic integers with given norm and/or trace, one needs to take higher-codimension slices.  For example, the minimal polynomial of a unit will have leading coefficient 1 and constant coefficient $\pm1$.  But one quickly discovers that these higher-dimensional slices have volumes that are, in general, no longer polynomial in $T$.  Rather than trying to explicitly calculate these volumes, we depart from the methods of earlier works, and instead approximate the volumes of such slices.  

When we cut a dilate $T\mcU_d$ by a certain kind of linear space, then as $T$ grows the slices look more and more like a lower-dimensional unit star body; this will be explained in Section \ref{volsec}.  This explains the appearance of the volume $V_d$ in all of our asymptotic counts.  We also use a careful analysis of the boundary of $\mcU_d$ to show that the above convergence happens relatively fast; this makes our approximations precise enough to obtain algebraic number counts with good power-saving error terms.

We state here our main result on counting polynomials.  For non-negative integers $m$, $n$, and $d$ with $0 < m+n \leq d$, and integer vectors $\vl \in \ZZ^m$ and $\vr \in \ZZ^n$, let $\mcM(d,\vl,\vr,T)$ denote the number of polynomials $f$ of the form
\begin{equation}\label{specpoly}
f(z) = \ell_0 z^d + \cdots + \ell_{m-1}z^{d-(m-1)}+x_mz^{d-m} + \cdots + x_{d-n}z^n + r_{d-n+1}z^{n-1} + \cdots + r_d
\end{equation}
with Mahler measure at most $T$, where $x_m,\dots,x_{d-n}$ are integers.  Let $g = d-m-n$.

Combining our volume estimates with a counting principle of Davenport, we obtain the following.

\begin{theorem}\label{mainthm}
For all $0 <m+n \leq d$, $\vl \in \ZZ^m$, and $\vr \in \ZZ^n$, as $T \to \infty$ we have
\begin{equation}
\mcM(d,\vl,\vr,T) = V_g\cdot T^{g+1} +  O(T^{g}).
\end{equation}
%where the implied constant depends only on $d$, $\vl$, and $\vr$.
\end{theorem}
\noindent Here the implied constant depends on $d, \vl,$ and $\vr$.

Now we briefly discuss the methods used in the second half of the paper to prove our explicit results, and how these results fit in with the literature.  Chern and Vaaler's \cite[Theorem 3]{chernvaaler}, which is the main ingredient in (\refeq{mvthm}), gives an asymptotic count of the number of integer polynomials of given degree $d$ and Mahler measure at most $T$.  The error term in this result contains a full power savings -- order $T^d$ against a main term of order $T^{d+1}$ -- but the implied constant in the error term is not made explicit.  They do produce an explicit error term of order $T^{d+1-1/d}$ in \cite[Theorem 5]{chernvaaler} using \cite[Theorem 4]{chernvaaler}, which is a quantitative statement on the continuity of the Mahler measure.
%}

%\textcolor{blue}{
Our Theorem \ref{genpolycount} below makes the constant in the error term of \cite[Theorem 3]{chernvaaler} explicit, using a careful study of the boundary of $\mcU_d$.  We apply the classical Lipschitz counting principle in place of the Davenport principle; the latter is not very amenable to producing explicit bounds.  Theorem \ref{moniccount} is the analogous result to Theorem \ref{genpolycount} for monic polynomials, and is obtained in a similar manner.  However, the application of the Lipschitz principle is more delicate in this case. We also prove an explicit version of our Theorem \ref{mainthm} counting polynomials with specified coefficients (Theorem \ref{slicecount}).  For this result we also apply \cite[Theorem 4]{chernvaaler}, and, reminiscent of Chern and Vaaler's application, this method yields an inferior power savings.
%}

We now describe the organization of the paper.  In Section \ref{starsec} we collect key facts about the unit star body $\mcU_d$, including a detailed discussion of its boundary.  In Section \ref{countingsec} we describe the counting principles we use to estimate the difference between the number of lattice points in a set and the set's volume.  In Section \ref{volsec} we estimate the volume of the sets in which we must count lattice points to prove Theorem \ref{mainthm}; this theorem is then proved in Section \ref{latticesec}.  In Section \ref{finalcountsec} we transfer our counts for polynomials to counts for various kinds of algebraic numbers, thereby proving Theorem \ref{maincor} and Corollaries \ref{unitcor}-\ref{normtracecor}.  This involves using a version of Hilbert's irreducibility theorem to account for reducible polynomials.  

The rest of the paper is devoted to obtaining explicit versions of these counts.  In Section \ref{cpebsec} we prove the aforementioned explicit version of \cite[Theorem 3]{chernvaaler} on counting polynomials of given degree and bounded Mahler measure, and in Section \ref{monicsec} we do the  same for the count of monic polynomials.  Section \ref{slicessec} contains a version of the general Theorem \ref{mainthm} with an explicit error term, at the cost of weaker power savings.  In Section \ref{sievingsec} we begin to convert our explicit counts of polynomials to explicit counts of minimal polynomials.  The main piece of this is showing that the reducible polynomials are negligible.  We follow the techniques for this used by Masser and Vaaler (sharper than the more general Hilbert irreducibility method described above), obtaining explicit bounds.  In Section \ref{exthmssec} we prove our final explicit results on counting algebraic numbers, including explicit versions of Masser and Vaaler's result (\refeq{mvthm}), Barroero's result (\refeq{barthm}), and Corollaries \ref{unitcor} and \ref{normcor}.  Finally, we include an appendix with some estimates for various expressions involving binomial coefficients which occur in our explicit error terms throughout the paper.  %These estimates are totally elementary, but a bit messy.  

%This mainly involves excluding reducible polynomials.  In general this is done using a form of Hilbert's irreducibility theorem, but in the cases of most interest we carry out a different approach.  This method is already found in the aforementioned work of Masser-Vaaler and probably originates with Schmidt, and allows us to improve power savings and obtain explicit bounds.  

\subsection*{Acknowledgments}
The authors would like to thank Antoine Chambert-Loir for useful correspondence related to Remark \ref{maninremark}, and Melanie Matchett Wood for useful comments on an early draft of this paper.

%2 
\section{The unit star body}\label{starsec}
In this section we discuss some properties of the unit star body
\begin{equation}
\mcU_d := \{\vw \in \RR^{d+1} ~\big|~ \mu(\vw)\leq 1\}.
\end{equation}
Since for all $f \in \RR[x]$ and $t \in \RR$ we have
\begin{equation}\label{dil}
\mu(tf) = |t|\mu(f),
\end{equation}
it's easy to see that $\mcU_d$ is in fact a (symmetric) star body.  Furthermore, $\mcU_d$ is compact; it is closed because $\mu$ is continuous\cite[Lemma 1]{Mahler61},  %\cite[Theorem 4]{chernvaaler} is a quantitative form of this fact that we will make use of later)
and we can see it is bounded by classical results that bound the coefficients of a polynomial in terms of its Mahler measure, for example the following (see \cite[p. 7]{mahler} and \cite[Lemma 1.6.7 and its proof]{bombierigubler}).
\begin{lemma}[Mahler]\label{coefbound}
Every polynomial $f(z) = w_0z^d + w_1z^{d-1} + \cdots +w_0 \in \CC[z]$ has coefficients satisfying
\begin{equation}\label{coeffbound}
|w_i| \leq {d \choose i} \mu(f), ~i = 0,\dots,d.
\end{equation}
Furthermore, we have the following double inequality comparing Mahler measure with the sup-norm of coefficients:
\begin{equation}\label{msup}
{d \choose{\lfloor d/2\rfloor}}^{-1} \|\vw\|_\infty \leq \mu(\vw) \leq \sqrt{d+1}\|\vw\|_\infty, ~\forall~\vw \in \RR^{d+1}.
\end{equation}
\end{lemma}

\subsection{Volumes}
As mentioned in the introduction, the exact volume of $\mcU_d$ was determined by Chern and Vaaler \cite[Corollary 2]{chernvaaler}:
\begin{equation}
V_d := \vol_{d+1}(\mcU_d) = 2^{d+1}(d+1)^s \prod_{j=1}^s \frac{(2j)^{d-2j}}{(2j+1)^{d+1-2j}},
\end{equation}
where $s = \lfloor(d-1)/2\rfloor.$  

We record some numerical information about the volume of $\mcU_d$.  We note that a result like the one below would follow quite easily from the asymptotic formula for $V_d$ given in \cite[(1.31)]{chernvaaler}.  However, this formula was given without proof and appears to contain an error.  We settle for a simpler result.
 
\begin{lemma}\label{volmax}
We have
\begin{align}
V_d \leq V_{15} &= \frac{2658455991569831745807614120560689152}{13904872587870848957579157123046875}\\ &= \frac{2^{121}}{3^{20}\cdot 5^9 \cdot 7^9 \cdot 11^6 \cdot 13^4}\approx 191.1888
\end{align}
for all $d \geq 0$, and
\begin{equation}
\lim_{d\to \infty}V_d = 0.
\end{equation}
\end{lemma}
\begin{proof}
Note using Stirling's estimates (see (\refeq{stir}) in the appendix) that for any positive integer $s$, we have
\begin{align}
\prod_{j=1}^s \left\{ \frac{2j}{2j+1}\right\} &= \frac{2^s s!}{(2s+1)!/(2^s s!)} = \frac{4^s s!^2}{(2s+1)!}\\
&\leq \frac{4^s(e^{1-s}s^{s+1/2})^2}{\sqrt{2\pi}e^{-2s-1}(2s+1)^{2s+3/2}} \leq \frac{4^s(e^{2-2s}s^{2s+1})}{\sqrt{2\pi}e^{-2s-1}(2s)^{2s+3/2}}\\
&\leq \frac{e^3 4^s s^{2s+1}}{\sqrt{2\pi} 4^s 2^{3/2} s^{2s+1} \sqrt{s}} \leq \frac{e^3}{4\sqrt{\pi s}}.
\end{align}

Suppose that $d$ is odd, so we may take $s =  \left \lfloor \frac{d-1}{2}\right \rfloor =  \left \lfloor \frac{(d+1)-1}{2}\right \rfloor.$ Then we have
\begin{align}
\frac{V_{d+1}}{V_d} &=  \frac{2^{d+2}(d+2)^s}{2^{d+1}(d+1)^s}\prod_{j=1}^s\left\{\frac{(2j)^{d+1-2j}}{(2j)^{d-2j}}\right\} \prod_{j=1}^s\left\{\frac{(2j+1)^{d+1-2j}}{(2j+1)^{d+2-2j}}\right\}\\
&= 2\left(\frac{d+2}{d+1}\right)^s \prod_{j=1}^s \left\{ \frac{2j}{2j+1}\right\} 
\leq \left(\frac{d+2}{d+1}\right)^s \cdot \frac{e^3}{2\sqrt{\pi s}}.\\
\end{align}
If $d$ is even and $s =  \left \lfloor \frac{d-1}{2}\right\rfloor = \frac{d}{2}-1$, then $\left \lfloor \frac{(d+1)-1}{2}\right\rfloor  = s+1$, and then we have
\begin{align}
\frac{V_{d+1}}{V_d}&=\frac{2^{d+2}(d+2)^{s+1}}{2^{d+1}(d+1)^s}\cdot \frac{d}{(d+1)^2}\prod_{j=1}^s\left\{\frac{(2j)^{d+1-2j}}{(2j)^{d-2j}}\right\} \prod_{j=1}^s\left\{\frac{(2j+1)^{d+1-2j}}{(2j+1)^{d+2-2j}}\right\}\\
&= 2\frac{(d+2)^{s}}{(d+1)^{s}}\cdot \frac{d^2+2d}{d^2+2d+1}\cdot \prod_{j=1}^{s}\left\{ \frac{2j}{2j+1}\right\} \leq \left(\frac{d+2}{d+1}\right)^{s} \cdot \frac{e^3}{2\sqrt{\pi s}}.\\
%%
%&= \frac{2(d+2)(d-2)^3}{(d-1^4)}\prod_{m=1}^s \left\{ \frac{2m}{2m+1}\right\} 
%=2 \frac{d^4-4d^3+16d-16}{d^4-4d^3+6d^2-4d+1} \prod_{m=1}^s \left\{ \frac{2m}{2m+1}\right\} \\
%&\leq 2 \prod_{m=1}^s \left\{ \frac{2m}{2m+1}\right\} \leq \frac{e^3}{2 \sqrt{\pi s}}.
\end{align}
In either case, the ratio of successive terms tends to zero, so in fact $V_d$ decays to zero faster than exponentially, proving the second claim of our lemma.  For the first claim, it suffices to compute enough values of $V_d$.  We see the maximum is attained at $d=15$, as advertised.
\end{proof}

For any $T \geq 0$, by (\refeq{dil}) we have that
\begin{equation}
\vol_{d+1}\left(\{\vw \in \RR^{d+1} ~\big|~ \mu(\vw) \leq T\}\right)=\vol_{d+1}(T\mcU_d) = V_d \cdot T^{d+1}. 
\end{equation}
Chern and Vaaler (see \cite[equation (1.16)]{chernvaaler}, corrected as in \cite[footnote on p. 38]{barroero14}) also computed the volume of the ``monic slice'' 
\begin{align} \label{monicdef}
\mcW_{d,T} &:= \{(w_0,\dots,w_d) \in T\mcU_d ~\big|~ w_0 = 1\}.
\end{align}
They showed:
\begin{align}
\vol_d\left(\mcW_{d,T}\right) = p_d(T) &:= \mcC_d2^{-s}\{s!\}^{-1} \sum_{m=0}^s (-1)^m(d-2m)^s{s \choose m}T^{d-2m}, \label{fd}
\end{align}
where again
\begin{equation}
s = \left\lfloor\frac{d-1}{2}\right\rfloor,~\textup{and}~\mcC_d = 2^d \prod_{j=1}^s \left(\frac{2j}{2j+1}\right)^{d-2j}.
\end{equation}
Note that, since $p_d(T)$ is a polynomial in $T$, we automatically have (carefully inspecting the leading term):
\begin{equation}
\vol_d\left(\mcW_{d,T}\right) = V_{d-1} \cdot T^d + O(T^{d-1}).
\end{equation}
For other slices besides the monic one, we will have to work harder (in Section \ref{volsec}) to obtain such power savings.  Along the way, it will become clear why the leading coefficient takes the form it does.

\subsection{Semialgebraicity}
  
Next we establish a qualitative result we will need in proving Theorem \ref{mainthm}.  A (real) \emph{semialgebraic set} is a subset of euclidean space which is cut out by finitely many polynomial equations and/or inequalities, or a finite union of such subsets.  Recall that semialgebraic sets are closed under finite unions and intersections, and they are closed under projection by the Tarski-Seidenberg theorem \cite[Theorem 1.5]{bierstonemilman}.
\begin{lemma}\label{semialglemma}
The set $\mcU_d\subset \RR^{d+1}$ is semialgebraic.
\end{lemma}
\begin{proof} 
Our proof is similar to that of \cite[Lemma 4.1]{barroero14}.  For $j = 0, \dots, d$, we wish to define a semialgebraic set $S_j \subset \RR^{d+1}$ corresponding to degree $j$ polynomials in $\mcU_d$.  We start by constructing auxiliary subsets of $\RR^{d+1} \times \CC^j$ corresponding to the polynomials' coefficients and roots, where $\CC$ is identified with $\RR^2$ in the obvious way.  We define
\begin{align}
S_j^0 = \{(0, \dots,0,w_{d-j},\dots, w_d,\alpha_1,\dots,\alpha_j) \in \RR^{d+1}\times \CC^j ~\big|~ w_{d-j} \not = 0,~\textup{and}&\\
w_{d-j}z^j +w_{d-j+1}z^{j-1} + \cdots + w_d = w_{d-j}(z-\alpha_1) \cdots(z-\alpha_j)&
\},
\end{align}
where the equalities defining the set are given by equating the real part of each elementary symmetric function in the roots $\alpha_1,\dots,\alpha_j $ with the corresponding coefficient $w_i$, and setting the imaginary part to zero.  To enforce $\mu((0, \dots,0,w_{d-j},\dots, w_d)) \leq 1$, we define $S_j^1$ to comprise those elements of $S_j^0$ such that all products of subsets of $\{\alpha_1,\dots,\alpha_j\}$ are less than or equal to $1/|w_{d-j}|$ in absolute value.  Finally, we let $S_j$ be the projection of $S_j^1$ onto $\RR^{d+1}$.  Now simply note that
\begin{equation}
\mcU_d = \{0\} \cup \bigcup_{j=0}^d S_j.
\end{equation}
\end{proof}

\begin{remark}\normalfont
Note that for any $T>0$ the dilation $T\mcU_d$ is also semialgebraic, and is defined by the same number of polynomials (and of the same degrees) as is $\mcU_d$.
\end{remark}

%%%%%
\subsection{Boundary parametrizations}\label{paramsec}

Next we describe the parametrization of the boundary of $\mcU_d$, which consists of vectors corresponding to polynomials with Mahler measure exactly 1.  The simple idea behind the parametrization is that such a polynomial is the product of a \emph{monic} polynomial with all its roots inside (or on) the unit circle, and a polynomial with constant coefficient $\pm 1$ and all its roots outside (or on) the unit circle.  Recall that $\mcU_d$ is a compact, symmetric star body in $\RR^{d+1}$.  The parametrization is described in \cite[Section 10]{chernvaaler}.  We briefly summarize the key points here.  The boundary $\partial \mcU_d$ is the union of $2d+2$ ``patches'' $\mcP_{k,d}^{\vep}$, $k = 0,\dots,d$, $\varepsilon = \pm 1$.  The patch $\mcP_{k,d}^{\vep}$ is the image of a certain compact set $\mcJ_{k,d}^{\vep}$ under the map 
\begin{equation}
b_{k,d}^{\vep}:\RR^k \times \RR^{d-k} \to \RR^{d+1},
\end{equation}
defined by
\begin{equation}\label{bdef}
b_{k,d}^{\vep}\big((x_1,\dots,x_k),(y_0,\dots,y_{d-k-1})\big) = B_{k,d}\big((1,x_1,\dots,x_k),(y_0,\dots,y_{d-k-1},\varepsilon)\big),
\end{equation}

\begin{equation}
B_{k,d}\big((x_0,x_1,\dots,x_k),(y_0,\dots,y_{d-k})\big) = (w_0,\dots,w_d),
\end{equation}
with
\begin{equation}\label{ws}
w_i = \sum_{l=0}^k \sum_{\shortstack{$m=0$\\$l+m=i$}}^{d-k} x_ly_m, \quad\quad i=0,\dots,d.
\end{equation}
Note that this simply corresponds to the polynomial factorization
\begin{align}
w_0z^d + \cdots + w_d = (x_0z^k + \cdots+x_k)\cdot (y_0z^{d-k} + \cdots +y_{d-k}).
\end{align}

The sets $\mcJ_{k,d}^{\vep}$ are given by 
\begin{equation}
\mcJ_{k,d}^{\vep} = J_k \times K_{d-k}^{\vep} \subseteq \RR^k \times \RR^{d-k},
\end{equation}
where
\begin{align}
J_k = \{\vx \in \RR^k ~\big|&~ \mu(1,\vx) = 1\},~\textup{and}~\label{Jdef}\\K_{d-k}^{\vep} = \{\vy \in \RR^{d-k} ~\big|&~ \mu(\vy,\vep) =1\}.
\end{align}

It will also be useful in Section \ref{monicsec} to have a parametrization of $\partial \mcW_{d,T}$, the boundary of a monic slice (see (\refeq{monicdef})), along the lines of that given for $\partial\mcU_d$ above.  Consider a monic polynomial 
\begin{equation}
 f(z) = z^d + w_1z^{d-1} + \cdots + w_d \in \RR[z],
\end{equation}
having Mahler measure equal to $T > 0$ and roots $\alpha_1,\dots,\alpha_d \in \CC$.  We note that such a polynomial can be factored as $f(z) = g_1(z)g_2(z)$, where $g_1$ and $g_2 \in \RR[z]$ are monic, $\mu(g_1) = 1$ (forcing $\mu(g_2) = T$), the constant coefficient of $g_2$ is $\pm T$, and where $\deg(g_1) = k  \in \{0,\dots,d-1\}$.  To do this, we simply let 
\begin{equation}
g_1(z) = \prod_{|\alpha_i| \leq 1}(z-\alpha_i), ~\textup{and}~ g_2(z) = \prod_{|\alpha_i| > 1}(z-\alpha_i).
\end{equation}
It is easy to check that $g_1$ and $g_2$ have the desired properties.  For $k=0,\dots,d-1$, we let $J_k$ be as in (\refeq{Jdef}), and let
\begin{align}
Y_{d-k}^{\vep T} &= \{\vy \in \RR^{d-k-1}~\big|~\mu(1,\vy,\vep T) = T\},~\textup{and}\\
\mcL_{k,d}^{\vep T} &= J_k \times Y_{d-k}^{\vep T} \subseteq \RR^k \times \RR^{d-k-1},
\end{align} 
for each $k = 0,\dots,d-1$, $\vep = \pm 1$.  We also define
\begin{equation}\label{mpar1}
\beta_{k,d}^{\vep T}\big((x_1,\dots,x_k),(y_1,\dots,y_{d-k-1})\big) = B_{k,d}^{\vep}\big((1,x_1,\dots,x_k),(1,y_1,\dots,y_{d-k-1},\varepsilon T)\big),
\end{equation}
similarly to (\refeq{bdef}).

We have that $\partial \mcW_{d,T}$ is covered by the $2d$ ``patches''  
\begin{equation}\label{mpar2}
\beta_{k,d}^{\vep T}\left(\mcL_{k,d}^{\vep T}\right).
\end{equation}

%
%Fix for the moment $L \in \{0,\dots,N\}$ and $\vep \in \{\pm1\}$, and set $\mcJ = \mcJ_{L,N}^{\vep}.$  Notice that, for any $n \in \{0,\dots,N\}$, by (\refeq{ws}) we have
%\begin{equation}\label{gradest}
%\|\nabla f_n (\vx,\vy)\|_\infty \leq \max \{1,\|(\vx,\vy)\|_\infty,\}.
%\end{equation}
%
%For any $n \in \{0,\dots,N\}$ and any $(\vx_1,\vy_1), (\vx_2,\vy_2) $, using (\refeq{boxbound}) and (\refeq{gradest})we have
%\begin{align}
%\big|Tf_n(\vx_1,\vy_1) - Tf_n(\vx_2,\vy_2)\big| &= T\big|f_n(\vx_1,\vy_1) - f_n(\vx_2,\vy_2)\big|\\
%&\leq T\cdot \sup_{(\vx,\vy) \in \mcJ} \|\nabla f_n(\vx,\vy)\|_2 \cdot \|(\vx_1,\vy_1)-(\vx_2,\vy_2)\|_2\\
%&\leq T \cdot \sqrt{N} \cdot \sup_{(\vx,\vy) \in \mcJ} \|(\vx,\vy)\|_\infty \cdot \sqrt{N}\cdot \|(\vx_1,\vy_1)-(\vx_2,\vy_2)\|_\infty\\
%&\leq T \cdot \sqrt{N} \cdot {N \choose \lfloor N/2 \rfloor}\cdot \sqrt{N}\cdot \|(\vx_1,\vy_1)-(\vx_2,\vy_2)\|_\infty\\ 
%&= N \cdot  {N \choose \lfloor N/2 \rfloor} \cdot T\cdot \|(\vx_1,\vy_1)-(\vx_2,\vy_2)\|_\infty. 
%\end{align}

%3
\section{Counting principles}\label{countingsec}
We'll need a counting principle of Davenport to estimate the number of lattice points in semialgebraic sets.  
%the fact that the set in which we wish to count lattice points is semialgebraic. 

\begin{theorem}[Davenport]\label{davenportprinciple}
Let $S$ be a compact, semialgebraic subset of $\RR^n$ defined by at most $k$ polynomial equalities and inequalities of degree at most $l$.  Then the number of integer lattice points contained in $S$ is equal to
\begin{equation}
\vol_n(S) + O(\max\{\overline{\vol}(S), 1\}),
\end{equation}
where $\overline{\vol}(S)$ denotes the maximum, for $m=1, \dots, n-1$, of the volume of the projection of $S$ on the $m$-dimensional coordinate space given by setting any $n-m$ coordinates equal to zero.  The implicit constant in the error term depends only on $k$, $l$, and $n$.
\end{theorem}

\begin{remark}\normalfont
This follows from the main theorem of \cite{davenport}, as described immediately after its statement. (The argument for this reduction was corrected in \cite{davcor}.) Davenport's principle has been generalized in a couple directions, to allow for lattices other than the standard integer lattice \cite[(1.2)]{barroerowidmer}, and to apply to sets definable in any $o$-minimal structure \cite[Theorem 1.3]{barroerowidmer}, of which semialgebraic sets are but one example.  However, the above version will suffice for our purposes.
\end{remark}

For our explicit error estimates 
we will use a different %we will instead use a more fundamental
counting principle, namely a refinement of the classical Lipschitz counting principle due to Spain \cite{spain}.  The classical principle allows one to estimate the difference between the number of lattice points in a set and the set's volume: one uses that the boundary is parametrized by finitely many Lipschitz maps, and that a Lipschitz map sends a cube in the domain into a cube in the codomain.  In our case it will be convenient to use ``tiles'' other than cubes in the domain.  This could be achieved by precomposing the maps with other maps which cover our tiles with the images of cubes, but we feel the following alternative formulation is intuitive and less awkward in application.  %For a subset $S \subseteq \RR^n$, we write $Z(S) = \#\{S \cap \ZZ^n\}$ for the number of integer lattice points in $S$.
 
\begin{theorem}\label{ourspain}
Let $S \subset \RR^n$ be a set whose boundary $\partial S$ is contained in the images of finitely many maps $\phi_i: J_i \to \RR^n$, where $\mcI$ is a finite set of indices and each $J_i$ is a set.  %, and each $J_i$ is a subset of $\RR^{n-1}$.  
For each $i \in \mcI$, assume that $J_i$ can be covered by $m_i$ sets $T_{i,1}, \dots, T_{i,m_i}$, with the property that for each $j$ the image $\phi_i(T_{i,j})$ is contained in a translate of $[0,1]^n$ inside $\RR^n$.  Then
\begin{equation}
\left|\#(S \cap \ZZ^n) - \vol_n(S)\right| \leq 2^n \sum_{i \in \mcI} m_i.
\end{equation}
\end{theorem}
 
\begin{proof}
We follow the ``every other tile'' approach of \cite{spain}.  
%By the usual argument, 
The number of lattice points in $S$ differs from the volume of $S$ by at most 
%$R(S)$, which we write for 
the number of integer vector translates of the half-open unit tile $[0,1)^n \subseteq \RR^n$ that meet the boundary $\partial S$.  Consider the set $\mcE$ of tiles which are \emph{even} integer vector translates of $[0,1)^n$; it is clear that any translate of $[0,1]^n$ meets exactly one such tile.  Since $\partial S$ is contained in at most $\sum_{i \in \mcI} m_i$ translates of $[0,1]^n$, this means that at most that many tiles from $\mcE$ meet $\partial S$.  But $\RR^n$ is partitioned by $2^n$ sets of tiles which, like $\mcE$, are made up of ``every other tile.''  (Explicitly, these sets are of the form $\mcE + \vec v$, where $\vec v$ is a vector of 0's and 1's.)  The bound claimed in the theorem follows.
\end{proof}

 %4
\section{Volumes of slices of star bodies}\label{volsec}
We keep all the notation established just before Theorem \ref{mainthm} in the introduction, so $d, m, n, \vl = (\ell_0,\dots,\ell_{m-1}) \in \ZZ^m,$ and $\vr = (r_{d-n+1},\dots,r_d) \in \ZZ^n$\footnote{For this section we could take $\vl$ and $\vr$ to be \emph{real} vectors, but this will not be important for our results.} are fixed, and again we set $g = d-m-n.$  Let $T$ be a positive real number.  The primary step in proving Theorem \ref{mainthm} is to estimate the volume of the slice 
\begin{align}
\mcS(T) &= \mcS_{\vl,\vr}(T) :=  \{\vw = (w_0,\dots,w_d) \in \RR^{d+1}~\big|~\mu(\vw) \leq T; \\
&\hspace{200pt}w_i = \ell_i,~\textup{for}~ i = 0,\dots,m-1;~\textup{and} \\
&\hspace{199pt}w_j = r_j, ~\textup{for}~ j = d-n+1,\dots,d\}\hspace{20pt}\label{ST}
\end{align}
as $T$ grows.  Specifically, we show the following.
\begin{theorem}\label{volest}
We have
\begin{equation}
\vol_{g+1}(\mcS(T)) = V_gT^{g+1} + O(T^g), ~\textup{as}~ T \to \infty.
\end{equation}
\end{theorem}
We won't obtain an explicit error estimate of this strength, but in Section \ref{slicessec} we will discuss how to obtain an explicit error term of order $T^{g+1-\frac{1}{d}}$.

The idea of the proof of Theorem \ref{volest} is as follows.  Because $\mu(T\vw) = T\mu(\vw)$ for all $T\geq0,$ and all $\vw \in \RR^{d+1}$, we have
\begin{equation}
\{\vw \in \RR^{d+1} ~\big|~ \mu(\vw) \leq T\} = T\{\vw \in \RR^{d+1} ~\big|~ \mu(\vw) \leq 1\} = T\mcU_d.  
\end{equation}
Let 
\begin{equation}
\vv = (\ell_0,\dots,\ell_{m-1},0,\dots,0,r_{d-n+1},\dots,r_d) \in \RR^{d+1},
\end{equation}
and for each $t \in [0,\infty)$, set 
\begin{equation}\label{Wt}
W_t := t\vv+\Span\{e_m, e_{m+1},\dots,e_{d-n}\} \subset \RR^{d+1},
\end{equation}
where $e_0,e_1,\dots,e_d$ are standard basis vectors for $\RR^{d+1}$.  Then for $T>0$ we have
%\begin{align}
%S(T) &= \{\vw = (w_0,\dots,w_d) \in T\mcU_d ~\big|~&w_i = \ell_i,~\textup{for}~ i = 0,\dots,m-1;~\textup{and}\\
%&&w_j = r_j, ~\textup{for}~ j = d-n+1,\dots,d\}\\
%&= T\{\vw = (w_0,\dots,w_d) \in T\mcU_d ~\big|~&w_i = \ell_i/T,~\textup{for}~ i = 0,\dots,m-1;~\textup{and}\\
%&&w_j = r_j/T, ~\textup{for}~ j = d-n+1,\dots,d\}
%\end{align}
\begin{align}\label{STis}
\mcS(T) = W_1 \cap T\mcU_d = T\left(W_{1/T} \cap \mcU_d\right),
\end{align}
and since $W_{1/T}$ is $(g+1)$-dimensional, this means
\begin{equation}\label{dory}
\vol_{g+1}(\mcS(T)) = T^{g+1}\vol_{g+1}\left(W_{1/T}\cap \mcU_d\right).
\end{equation}
Letting $t = 1/T$, we should expect that
\begin{equation}
\vol_{g+1}\left(W_{1/T}\cap \mcU_d\right) = \vol_{g+1}\left(\mcU_d \cap \left(W_0 + t\vv\right)\right) \to \vol_{g+1}\left(\mcU_d \cap W_0\right),~\textup{as}~t \to 0, 
\end{equation}
unless the boundary of $\mcU_d$ were to intersect with $W_0$ in an unusual way; for example, if $\mcU_d$ were a cube and $W_0$ was a plane containing one of the faces.  This basic idea of using continuity of volumes of slices appears in the proof of \cite[Theorem 1.5]{sinclair}.  We will show below that $\vol_{g+1}\left(\mcU_d \cap W_0\right) = V_g,$ whence the main term in the statement of Theorem \ref{volest}.  
%To obtain a power-saving error term, one must show that the boundary of $\mcU_d$ is well-behaved in some sense.  Since $\mcU_d$ is a semialgebraic set, it can be shown using abstract arguments (e.g. \cite[Th\'{e}or\`{e}me 1]{lionrolin}, cf. \cite[Corollary 6.4]{parusinski})  that \emph{some} fractional power-savings can be achieved.  In fact we obtain an \emph{explicit} error term with savings $1/d$ using a result of Chern and Vaaler (see Section \ref{slicessec}).  
We'll obtain a full power savings by showing that the boundary of $\mcU_d$ is never tangent to $W_0$.\footnote{As an exercise to see why tangency is a problem, consider the length of cross-sections of a disk as the cross-sections slide toward a tangent line.}  

\begin{proposition}\label{laurent}
Let $S \subset \RR \times \RR^N$ be a compact set % which is homeomorphic to the closed $(g+1)$-ball, 
bounded by finitely many smooth hypersurfaces $H_i, i = 1,\dots, m$.  
%Suppose that none of these hypersurfaces is tangent to the hyperplane $\{0\} \times \RR^n$.
Assume each boundary component $H_i \cap \partial S$ has smooth intersection with (i.e. is not tangent to) the hyperplane $\{0\}\times \RR^N$, and that these boundary components $H_i \cap \partial S$ have pairwise disjoint interiors.
Then
\begin{equation}
V(t) := \vol_N \left(S\cap(\{t\}\times \RR^N)\right)
\end{equation} 
%admits a power series expasion in $t$ as $t \to 0^+$.  In particular,
satisfies
\begin{equation}
V(t) = V(0) + O(t), ~\textup{as}~t \to 0^+.
\end{equation}
\end{proposition}

\begin{proof}
We denote points in $\RR \times \RR^N$ by $(x,y_1,\dots,y_N)$.  For each $t \geq 0$, let $S_{[0,t]} = S \cap ([0,t] \times \RR^N)$, and let $S_t = S \cap (\{t\} \times \RR^N)$.  Let $F$ denote the constant vector field $(1,0,\dots,0)$ on $\RR \times \RR^N$.  By the divergence theorem, we have
\begin{equation}
\oint_{\partial S_{[0,t]}} F \cdot d\vs = \int_{S_{[0,t]}} \nabla \cdot F ~d\vol_{N+1} = \int_{S_{[0,t]}} 0 ~d\vol_{N+1} = 0,
\end{equation}
where the first integral is with respect to the surface measure with outward normal.  Note that our assumption that $\{0\} \times \RR^N$ is not tangent to any of the $H_i$ means that neither is the parallel hyperplane $\{t\}\times \RR^{N}$ for $t$ sufficiently small.  Let $R_t =  ([0,t] \times \RR^N) \cap \partial S$, and note that, as long as $t$ is small enough to avoid the aforementioned tangencies, the boundary of $S_{[0,t]}$ decomposes into three pieces with disjoint interiors as follows:
\begin{equation}
\partial S_{[0,t]} = S_0 \cup S_t \cup R_t.
\end{equation}
and so we have
\begin{align}
0 &= \oint_{\partial S_{[0,t]}} F \cdot d\vs = \int_{S_0} F \cdot d\vs + \int_{S_t} F \cdot d\vs + \int_{R_t} F \cdot d\vs \\
&= -V(0) + V(t) + \int_{R_t} F \cdot d\vs,
\end{align}
where
\begin{align}
\hspace{0.25in}\int_{R_t} F \cdot ~d\vs = \sum_i \int_{H_i \cap R_t} F\cdot ~d\vs.
\end{align}
Now we must show that 
\begin{equation}\label{stufff}
|V(t) - V(0)| = \left|\int_{R_t} F \cdot d\vs\right| = O(t).
\end{equation}

Since $S$ is compact, the set $R_t$ is contained in a ``pizza box'' $%B_t := 
[0,t] \times [-M,M]^N$ for some positive number $M$ independent of $t$. 
Fix $i \in \{1,\dots,m\}$.  By assumption, $H_i\cap \partial S$ is not tangent to the hyperplane $\{x=0\}$, but since $H_i$ is smooth and we're working in a compact set, we know $H_i \cap \partial S$ is not tangent to $\{x=t\}$ for any $t$ sufficiently small.  
%we call any hyperplane parallel to this one ``vertical.''  
This means that, by the implicit function theorem, for $t$ sufficiently small and any point $P \in H_i \cap R_t$, we have that $H_i$ coincides in an open subset $U \subseteq H_i \cap R_t$ containing $P$ with the graph of a function $y_r = f(x,y_1,\dots,\hat y_r,\dots,y_N)$ for some $r \in \{1,\dots,N\}$ which depends on $P$.  %Our notation is meant to convey that $f$ does not depend on $y_r$, and s
So we have $f:V \to [-M,M],$ where $V$ is an open subset of $[0,t]\times [-M,M]^{N-1}$.  Letting $\vn$ denote the outward unit normal, we have
\begin{equation}\label{it}
\int_U F \cdot d\vs = \int_U F \cdot \vn ~ds =  \int \cdots \int_V  \mp \frac{\partial f}{\partial x} ~dx dy_1\cdots \hat{dy_r}\cdots dy_N,
\end{equation}  
where the sign in the final integral is $-$ or $+$ depending on whether $\vn$ is an upward or downward normal to the graph of $f$, respectively.

By our non-tangency assumption again, the partial derivative $\frac{\partial f}{\partial x}$ is bounded in absolute value inside our pizza box by a constant $K$ which does not depend on $U,$ $i,$ or $t$ as $t \to 0$. %If $U$ is one of these neighborhoods, we see that 
By compactness, finitely many of these neighborhoods $U$ cover $H_i \cap R_t$, and the number of neighborhoods required -- call this number $n$ -- can be chosen independent of $t$ or $i$.  Using (\refeq{it}), we estimate the integral in (\refeq{stufff}) as follows:

\begin{align}
\left|\int_{R_t} F \cdot d\vs\right| &\leq \sum_{i=1}^m \left|\int_{H_i \cap R_t} F \cdot d\vs\right| \leq \sum_{i=1}^m \int_{H_i \cap R_t} \left|F \cdot\vn\right| ds \leq \sum_{i=1}^m \sum_U \int_U \left|F\cdot \vn\right|~ds\\%\leq \sum_{i=1}^m\sum_U\left|\int_U F \cdot d\vs\right|\\ 
%\leq\sum_{i=1}^m \sum_U\left|\int \cdots \int_V  -\frac{\partial f}{\partial x} ~dx dy_1\cdots \hat{dy_r}\cdots dy_N\right|\\
&\leq \sum_{i=1}^m\sum_U\int_{-M}^M \cdots \int_{-M}^M \int_0^t \left|\frac{\partial f}{\partial x}\right| ~dx dy_1\cdots \hat{dy_r}\cdots dy_N \\
&\leq m\cdot n\cdot [(2M)^{N-1}t ]K = O(t). 
\end{align}\end{proof}

Now we verify that the boundary of $\mcU_d$ satisfies the hypotheses of Proposition \ref{laurent}.  We refer to the parametrization of said boundary described in Section \ref{starsec}, and follow that notation.  As noted in \cite[Section 10]{chernvaaler}, the condition of the boundary components having disjoint interiors is satisfied here -- this can be readily verified directly from the description of the parametrization.  Let $H = H_{k,d}^{\vep}$ be one of the hypersurfaces which bound $\mcU_d$.  The hypersurface $H$ is the image of $\RR^k \times \RR^{d-k}$ under the map $b = b_{k,d}^{\vep}$ described in (\refeq{bdef}).

\begin{proposition}\label{notan}
Let $\vv = (\ell_0,\dots,\ell_{m-1},0,\dots,0,r_{d-n+1},\dots,r_d) \in \RR^{d+1}$, and let 
\begin{align}
W_0 &= \spann\{e_m,e_{m+1},\dots,e_{d-n}\},~\textup{and}\\
W &= \spann\{\vv, e_m, e_{m+1},\dots,e_{d-n}\},
\end{align}
where $e_0,e_1,\dots,e_d$ are standard basis vectors for $\RR^{d+1}$.   Then $W_0$ is not tangent to $H \cap W$ at any point.
\end{proposition}

Let $P = (w_0,\dots,w_d) = b(x_1,\dots,x_k,y_0,\dots,y_{d-k-1})$ be a point on $H \cap W_0$.   We will break up the proof of this proposition into three lemmas.

\begin{lemma} \label{intersectionconstraints}
The subspace $W_0$ does not meet $H$ unless 
\begin{equation}\label{lrange}
n \leq k \leq d-m. 
\end{equation}
If those inequalities hold, then we have 
\begin{align}\label{vanish}
y_0 = \cdots = y_{m-1} = x_{k-n+1} = \cdots = x_k = 0.
\end{align}
\end{lemma}

\begin{proof} 
Suppose the inequalities are satisfied.  We'll prove vanishing of the parameters $y_i$, by induction on $0 \leq i \leq m-1$.  If $m=0$, there's nothing to prove.  Otherwise, for the base case $i=0$, by the definition of $W_0$ we have $w_0=0$, but also $w_0=y_0$ by the definition of $b$ in (\refeq{bdef}).  For arbitrary $i$, we again have $w_i=0$, while by the definition of $b$, every summand in the formula for $w_i$ is of the form $x_{i-j}y_j$ for $j < i$, except for the summand $y_i$.  Thus we're done by induction.  Essentially the same proof works for the vanishing of $x_{k-n+1},\dots,x_k$.

However, if $n > k$, then the above argument would imply that $x_0=0$, but we know $x_0=1$, a contradiction.  Similarly, if $k > d - m$, the above would give $0=y_{d-k}=\epsilon$, also a contradiction. 
\end{proof}

\begin{lemma}
The tangent space $T_P(H)$ of $H$ at $P$ is the row space of the following $d \times (d+1)$ matrix, where the first $(d-k)$ rows represent the tangent vectors $\left(\frac{\partial w_0}{\partial y_j},\dots, \frac{\partial w_d}{\partial y_j}\right),~ j = 0,\dots,d-k-1$, and the last $k$ rows represent the tangent vectors $\left(\frac{\partial w_0}{\partial x_i},\dots, \frac{\partial w_d}{\partial x_i}\right),~i = 1,\dots,k$.  Let $q = d-k-1$ for ease of reading.
\setcounter{MaxMatrixCols}{20}
\begin{equation}
(Db)^T =
\begin{bmatrix}
1         & x_1    & x_2    & \cdots  & \cdots & x_k    & 0      & 0       & \cdots     & \cdots&0\\
0         & 1        & x_1    & x_2     & \cdots &\cdots & x_k  & 0       & \cdots      & \cdots & 0\\
\vdots &~         &\ddots &\ddots  &\ddots  &~         &~      &\ddots &~              &~         &\vdots \\
\vdots &~         &~         & \ddots & \ddots &\ddots &~      &           &~   \ddots &~         &\vdots \\
0         &\cdots & \cdots & 0        & 1         & x_1    & x_2 & \cdots & \cdots & x_k     & 0\\
0        & y_0     & y_1    &\cdots  & \cdots & y_q    & \vep & 0       & \cdots &  \cdots &0\\
0        & 0         & y_0    & y_1    &\cdots  & \cdots & y_q    & \vep & 0        &  \cdots &0\\
\vdots &~         &&\ddots &\ddots  &~         &~      &\ddots &\ddots          &~         &\vdots \\
\vdots &~         &&&\ddots &\ddots  &~         &~      &\ddots &\ddots                  &0\\
0&0&\cdots&\cdots&0&y_0&y_1&\cdots&\cdots&y_q&\vep
\end{bmatrix}.
\end{equation}

\end{lemma}

\begin{lemma}\label{surj}
The projection of $T_P(H)$ onto $W_0^\perp$ is surjective.
\end{lemma}

\begin{proof}
Using Lemma \ref{intersectionconstraints}, the image of that projection contains the row space (in appropriate coordinates) of the following matrix, obtained by taking the first $m$ columns and first $m$ rows of the above matrix, as well as its last $n$ columns and last $n$ rows:

\begin{equation}
C:=
\begin{bmatrix}
A & 0\\
0 & B
\end{bmatrix},
\end{equation}

where
\begin{equation}
A=
\begin{bmatrix}
1         & x_1    & x_2    & \cdots  & x_{m-1}     \\
0         & 1        & x_1    & \cdots      & x_{m-2}\\
\vdots &\ddots         &\ddots &\ddots  &\vdots   \\
\vdots &~         &\ddots        & \ddots & x_1  \\
0         &\cdots & \cdots & 0        & 1             \\
\end{bmatrix}
\end{equation}
is an $m \times m$-matrix, and 

\begin{equation}
B=
\begin{bmatrix}
\vep         & 0    & \cdots    & \cdots  & 0     \\
y_q         & \vep        & \ddots    & ~      & \vdots\\
\vdots &\ddots         &\vep &\ddots  &\vdots   \\
y_{q-n+3} &\ddots         &\ddots         & \ddots & 0  \\
y_{q-n+2}         &\cdots & y_{q-1} & y_q        & \vep             \\
\end{bmatrix}
\end{equation}
is an $n \times n$-matrix.

Thus $C$ is a block diagonal matrix (we've used the vanishing of parameters described in (\refeq{vanish}) here) with determinant $\vep^n \neq 0$, so its row space is all of $W_0^\perp$.
\end{proof}

\begin{proof}[Proof of Proposition \ref{notan}]
We seek a tangent vector to $H$ at $P$ which is contained in $W \setminus W_0$.  By Lemma \ref{surj}, $T_P(H)$ surjects onto the positive-dimensional space $W_0^\perp$.  Since its kernel under this map is exactly $W_0$, a vector must exist as desired.	
\end{proof}

\begin{proof}[Proof of Theorem \ref{volest}]
We begin by noting that we may identify $ \mcU_d \cap W_0 \subseteq \RR^{d+1}$ with $\mcU_{g} \subseteq \RR^{g+1}$ as follows.

Define a map $\tau:\RR^{g+1} \to \RR^{d+1}$ by 
\begin{equation}
\tau(x_m,\dots,x_{d-n}) = (\underbrace{0,\dots,0}_m,x_m,\dots,x_{d-n},\underbrace{0,\dots,0}_n) \in W_0,
\end{equation}
which corresponds to multiplying the polynomial corresponding to the input by $z^n$.  Notice that this operation preserves the Mahler measure.  It's also clear that $\tau$ maps $\mcU_{g}$ isometrically onto $\mcU_d \cap W_0$, so we conclude that
\begin{equation}\label{fromtau}
\vol_{g+1}(\mcU_d \cap W_0) = \vol_{g+1}(\mcU_{g}) = V_g.
\end{equation}

Using Proposition \ref{notan}, we can apply Proposition \ref{laurent} to the set $\mcS = \mcU_d \cap W$, considered as a subset of $W  \cong \RR \times \RR^{g+1}$ (so we are setting $N = g+1$).  Here for $t \geq 0$ we have
\begin{equation}
 \mcS \cap \left(\{t\} \times \RR^{g+1}\right) = \mcU_d \cap W_t.
\end{equation}
Then Proposition \ref{laurent} gives
\begin{equation}
\vol_{g+1}\left(\mcU_d \cap W_{1/T}\right) = \vol_{g+1}\left(\mcU_d \cap W_0\right) + O(1/T).
\end{equation}

Now by (\refeq{dory}) and (\refeq{fromtau}) we have
\begin{align}
\vol_{g+1}(S(T)) &= \big(\vol_{g+1}\left(\mcU_d \cap W_0\big) + O(1/T) \right)\cdot T^{g+1}\\ 
&= V_g \cdot T^{g+1} + O(T^g),
\end{align}
completing our proof.
\end{proof}

%5
\section{Lattice points in slices: proof of Theorem \ref{mainthm}}\label{latticesec}
Now that we have an estimate for the volume of $\mcS(T)$, we want to in turn estimate the number of integer lattice points in $\mcS(T)$, via Theorem \ref{davenportprinciple}.   Note that this is the same as the number of integer lattice points of $S'(T)$, which will denote the projection of $\mcS(T)$ on $W_0 \cong \RR^{g+1}$. Note that $\vol(\mcS(T))=\vol(S'(T))$.

Since $\mcU_d$ is semialgebraic by Lemma \ref{semialglemma} (and thus $T\cdot \mcU_d$ as well), it is clear that the number and degrees of the polynomial inequalities and equalities needed to define $S'(T)$ are independent of $T$.  Thus to apply Theorem \ref{davenportprinciple}, it remains only to bound the volumes of projections of $S'(T)$ on coordinate planes.

For $\vw \in S'(T)$, by (\refeq{msup}) we have
\begin{equation}
\|\vw\|_\infty \leq \|(\vl, \vw, \vr)\|_\infty \leq {d\choose{\lfloor d/2\rfloor}}\mu(\vl,\vw,\vr) \leq {d\choose{\lfloor d/2\rfloor}}T,
\end{equation}
so $S'(T)$ is contained inside a cube of side length $2{d\choose{\lfloor d/2\rfloor}}T$ in $\RR^{g+1}$.  Thus for $j=1, \dots, g$, any projection of $S'(T)$ on a $j$-dimensional coordinate plane is contained inside a cube of side length $2{d\choose{\lfloor d/2\rfloor}}T$
in $\RR^j$, and thus has volume at most 
$(2{d\choose{\lfloor d/2\rfloor}}T)^{j}$, 
which is certainly $O(T^g)$ for $j=1, \dots, g$.

By Theorem \ref{davenportprinciple}, we now get $$\mcM(d,\vl,\vr,T) = \vol(S'(T)) + O(T^g),$$

and so by Theorem \ref{volest} we have $$\mcM(d,\vl,\vr,T) = V_g\cdot T^{g+1} +  O(T^{g}).$$
%\section{OLD SECTION 5}
%\input{latticeold.tex} 
  
%6
\section{Proofs of Theorem \ref{maincor} and corollaries}\label{finalcountsec}
In this section we transfer our counts for degree $d$ polynomials in Theorem \ref{mainthm} to the counts for degree $d$ algebraic numbers in Theorem \ref{maincor}. This only requires estimating the number of reducible polynomials, because the hypotheses imposed on the coefficients in Theorem \ref{maincor} ensure that the only irreducible polynomials we count are actually \emph{minimal} polynomials of degree $d$.
%don't allow for any irreducible polynomials to be counted other than legitimate minimal polynomials of degree $d$.  
We'll apply a version of Hilbert's irreducibility theorem to achieve the most general result, which is the last ingredient needed to prove Theorem \ref{maincor}.  However, in various special cases we work a little harder to improve the power savings, which will prove the sharper results of Corollaries \ref{unitcor} through \ref{normtracecor}.

We keep the notation and hypotheses of Theorem \ref{maincor}, fixing $d, m,n, \vl \in \ZZ^m,$ and $\vr \in \ZZ^n.$  Furthermore, we let $\mcM^{red}(d,\vl,\vr,T)$ denote the number of \emph{reducible} integer polynomials of the form 
\begin{equation}\label{specpoly2}
f(z) = \ell_0 z^d + \cdots + \ell_{m-1}z^{d-(m-1)}+x_mz^{d-m} + \cdots + x_{d-n}z^n + r_{d-n+1}z^{n-1} + \cdots + r_d,
\end{equation}
and as before we set $g = d-m-n$.
\begin{proposition}\label{HITprop}
We have
\begin{equation}\label{cohenbound}
\mcM^{red}(d,\vl,\vr,T) =  O\left(T^{g+\frac{1}{2}}\log T\right).
\end{equation}
\end{proposition}
\begin{proof}
One of our hypotheses is that, if $n >0$, then $r_d \not = 0$; that is, we don't want $f(z)$ to be divisible by $z$.  It's not hard to see that, under this hypothesis, the ``generic polynomial'' $f(x_m,\dots,x_{d-n},z)$ defined above is irreducible in $\ZZ[x_m,\dots,x_{d-n},z],$ by the following argument.  Suppose $f$ factors nontrivially as $f = f_1f_2.$  Since $f$ has degree 1 in $x_m$, without loss of generality $f_1$ has degree 1 in $x_m$ and $f_2$ has degree 0 in $x_m$.  Let $f_1 = g_1 x_m + g_2$, where $g_1$ and $g_2$ are in $\ZZ[x_{m+1},\dots,x_{d-n},z],$ so we have $f = f_2 g_1 x_m + f_2 g_2$, which means that $f_2 g_1 = z^{d-m}$.  We discover that $f_2$ is (plus or minus) a power of $z$, and so $f$ was divisible by $z$ all along.

Now our proposition follows immediately from a quantitative form of Hilbert's irreducibility theorem due to Cohen \cite[Theorem 2.5]{cohen}.  In the notation of the cited theorem, we are setting $r=1$, and $s = g+1$.  Cohen uses the $\ell_\infty$ norm on polynomials rather than Mahler measure, but these are directly comparable by (\ref{msup}).  It's worth noting that, as can be inferred from \cite[Section 2]{cohen}, the implied constant in (\refeq{cohenbound}) depends only on $d$, $g$, and $\|(\vl,\vr)\|_\infty$, and could in principle be effectively computed.
\end{proof}

In the situations of Corollaries \ref{unitcor} through \ref{normtracecor}, we can obtain stronger bounds.

\begin{proposition}\label{specialreducible}
For $d \geq 2$, and $r \in \ZZ \setminus \{0\}$, we have
\begin{equation}
\mcM^{red}(d,(1),(r),T) = O\left(T^{d-2}\right).
\end{equation}
For $d \geq 3$, $t \in \ZZ$, and $r \in \ZZ \setminus \{0\}$, we have
\begin{equation}
\mcM^{red}(d,(1,t),(r),T) =  O\left(T^{d-3}\right).
\end{equation}
For $d \geq 2$, $T \geq 1$, and $t \in \ZZ$, we have
%\begin{align}
%\mcM^{red}(d,(1,t),(),T) &=  O\left(\sqrt{T}\right),\hspace{16pt}\textup{if}~d=2,\label{d2t}\\
%\mcM^{red}(d,(1,t),(),T) &=  O\left(T\log T\right),~\textup{if}~d=3,~\textup{and} \label{d3t}\\
%\mcM^{red}(d,(1,t),(),T) &=O\left(T^{d-2}\right), \hspace{9.4pt}\textup{if}~d>3.\label{dg3t}
%\end{align}
\begin{align}\label{dz}
\mcM^{red}(d,(1,t),(),T)= \left\{
\begin{array}{l}
\displaystyle{O\left(\sqrt{T}\right),\hspace{16pt}\textup{if}~d=2,}\\
\displaystyle{O\left(T\log T\right),~\textup{if}~d=3,~\textup{and}}\\
\displaystyle{O\left(T^{d-2}\right), \hspace{9.4pt}\textup{if}~d>3.}
\end{array}\right.
\end{align}
\end{proposition}

We postpone the proof until Section \ref{sievingsec}, where we'll prove it with explicit constants.  For now, we show how Theorem \ref{maincor} and Corollaries \ref{unitcor} through \ref{normtracecor} follow from our results so far.

\begin{proof}[Proof of Theorem \ref{maincor} and Corollaries \ref{unitcor} through \ref{normtracecor}]
By Theorem \ref{mainthm} we have that 
\begin{equation}
\mcM(d,\vl,\vr,T) = V_g\cdot T^{g+1} + O(T^g).
\end{equation}
We write $\mcM^{irr}(d,\vl,\vr,T)$ for the corresponding number of \emph{irreducible} degree $d$ polynomials with specified coefficients.  Since $\vl$ is non-empty and $\ell_0\neq0$, we have
\begin{equation}\label{nicecase}
\mcM^{irr}(d,\vl,\vr,T) = \mcM(d,\vl,\vr,T) - \mcM^{red}(d,\vl,\vr,T).
\end{equation}
Applying Theorem \ref{mainthm} and Proposition \ref{HITprop}, we see that
\begin{equation}\label{this}
\mcM^{irr}(d,\vl,\vr,T) = V_g\cdot T^{g+1} + O(T^{g+\frac{1}{2}}\log T).
\end{equation}
By our assumption that the specified coefficients had no common factor, and that $\ell_0 > 0$, any irreducible polynomial counted will be a minimal polynomial.  Thus each of the degree $d$ irreducible polynomials $f$ we count corresponds to exactly $d$ algebraic numbers $\alpha_1,\dots,\alpha_d$ of degree $d$ and height at most $\mcH$, where $\mcH^d = T$, since $\mu(f) = H(\alpha_i)^d$ for $i = 1,\dots,d$.  In other words, we have
\begin{equation}
\mcN(d,\vl,\vr,\mcH) = d\mcM^{irr}(d,\vl,\vr,\mcH^d).
\end{equation}
Now Theorem \ref{maincor} follows from (\refeq{this}).

Corollaries \ref{normcor}, \ref{tracecor}, and \ref{normtracecor} follow similarly, by replacing the general upper bound for reducible polynomials in Proposition \ref{HITprop} with the sharper bounds in Proposition \ref{specialreducible}.  The count for units in Corollary \ref{unitcor} follows immediately from Corollary \ref{normcor}, since an algebraic number is a unit exactly if it is an algebraic integer with norm $\pm 1$.
\end{proof}

%7
\section{Counting polynomials: explicit bounds}\label{cpebsec}
Let $\mcM(\ld d,T)$ denote the number of polynomials in $\ZZ[z]$ of degree at most $d$ and Mahler measure at most $T$.  The following is an explicit version of \cite[Theorem 3]{chernvaaler}.  To condense notation, we define for each $d \geq 0$ the constants
\begin{align}
P(d) &= \prod_{j=0}^d {d \choose j}, ~\textup{and}\label{P}\\
A(d) &= \sum_{k=0}^d P(k)P(d-k).
\end{align}

\begin{theorem} \label{genpolycount}
For $d \geq 1$ and $T\geq 1$ we have
\begin{equation}
|\mcM(\ld d,T) - \vol(\mcU_d) T^{d+1}| \leq \kappa_0(d)T^d,
\end{equation}
where
 \begin{align}
 \kappa_0(d) &= 4^{d+1}A(d)\left(d{d \choose \lfloor d/2 \rfloor}+1\right)^d\label{1} \\
 &\leq 40 \sqrt[4]{2}\pi^{3/4}e^{-3} \cdot d^{-1/4}\cdot(4\sqrt{2}e^{3/2}\pi^{-3/2})^d \cdot (2\sqrt{e})^{d^2} \\
 &\leq 5.59 \cdot (15.01)^{d^2}.
 \end{align}
\end{theorem}
 
\begin{proof} 
We refer to the parametrization of the boundary of $\mcU_d$ detailed in Section \ref{paramsec}. 
The boundary $\partial (T\mcU_d)$ is parametrized by $2d+2$ maps of the form 
\begin{align}
Tb_{k,d}^{\vep} :\mcJ_{k,d}^{\vep} &\to \partial (T\mcU_d) \in \RR^{d+1},\label{par1}\\
Tb_{k,d}^{\vep}(\vx,\vy) &= \big(Tf_0(\vx,\vy),\dots,Tf_d(\vx,\vy)\big),\label{par2}
\end{align}
where 
\begin{equation}\label{par3}
f_i(\vx,\vy) := w_i\big((1,\vx),(\vy,\vep)\big),~ \textup{for} ~i=0,\dots, d,
\end{equation} 
and $w_i$ is as in (\refeq{ws}).

Fix for the moment $k \in \{0,\dots,d\}$ and $\vep \in \{\pm1\}$.  %, and set $\mcJ = \mcJ_{k,d}^{\vep}.$  
If $(\vx,\vy)$ lies in any $\mcJ_{k,d}^{\vep},$ then $\mu(1,\vx) = \mu(\vy,1) = 1$, and so by (\refeq{coeffbound}) we have $\|(\vx,\vy)\|_\infty \leq {d \choose \lfloor d/2 \rfloor}$, and so
\begin{equation}\label{boxbound}
\|(\vx,\vy)\|_2 \leq \sqrt{d} \|(\vx,\vy)\|_\infty \leq \sqrt{d} \cdot {d \choose \lfloor d/2 \rfloor}.
\end{equation}
Also, for any $i \in \{0,\dots,d\}$, by (\refeq{ws}) we have
\begin{equation}\label{gradest}
\|\nabla f_i (\vx,\vy)\|_\infty \leq \max \{1,\|(\vx,\vy)\|_\infty,\}.
\end{equation}
Now for any $i \in \{0,\dots,d\}$ and any $(\vx_1,\vy_1), (\vx_2,\vy_2) $, using (\refeq{boxbound}) and (\refeq{gradest}) we have
\begin{align}
\big|Tf_i(\vx_1,\vy_1) - Tf_i(\vx_2,\vy_2)\big| &= T\big|f_i(\vx_1,\vy_1) - f_i(\vx_2,\vy_2)\big|\\
&\leq T\cdot \sup_{(\vx,\vy) \in \mcJ} \|\nabla f_i(\vx,\vy)\|_2 \cdot \|(\vx_1,\vy_1)-(\vx_2,\vy_2)\|_2\\
&\leq T \cdot \sqrt{d} \cdot \sup_{(\vx,\vy) \in \mcJ} \|(\vx,\vy)\|_\infty \cdot \sqrt{d}\cdot \|(\vx_1,\vy_1)-(\vx_2,\vy_2)\|_\infty\\
&\leq T \cdot \sqrt{d} \cdot {d \choose \lfloor d/2 \rfloor}\cdot \sqrt{d}\cdot \|(\vx_1,\vy_1)-(\vx_2,\vy_2)\|_\infty\\ 
&= d \cdot  {d \choose \lfloor d/2 \rfloor} \cdot T\cdot \|(\vx_1,\vy_1)-(\vx_2,\vy_2)\|_\infty. 
\end{align}
 We obtain the Lipschitz estimate 
\begin{equation}\label{lip}
\|Tb_{k,d}^{\vep}(\vx_1,\vy_1) - Tb_{k,d}^{\vep}(\vx_2,\vy_2)\|_\infty \leq KT \cdot \|(\vx_1,\vy_1)-(\vx_2,\vy_2)\|_\infty,
\end{equation}
where $K = K(d) :=  d \cdot  {d \choose \lfloor d/2 \rfloor}\leq \sqrt{d} \cdot 2^d$.
%%%%%%%%

We now apply the Lipschitz counting principle from Section \ref{countingsec}.  Fix $T \geq 1$, so that $\lceil KT \rceil \leq KT+1 \leq (K+1)T.$ 
%We'll follow the proof of \cite[Theorem]{spain}, but we'll gain a bit by viewing the sets $\mcJ_{L,d}^{\vep}$ as inside boxes which aren't cubes.  
Since $Tb_{k,d}^{\vep}$ satisfies the Lipschitz estimate (\refeq{lip}), the image under $Tb_{k,d}^{\vep}$ of any translate of $\left[0,1/\lceil KT\rceil\right]^d$ is contained in a unit cube in $\RR^{d+1}$.  %By the usual argument, the number of lattice points in $T\mcU_N$ differs from the volume of $T\mcU_N$ by at most $R(N,T)$, which we write for the number of integer vector translates of the unit tile $[0,1)^{N+1} \subseteq \RR^{N+1}$ that meet the boundary $\partial (T\mcU_N).$  We'll write $R_{L,N}^{\vep}(T)$ for the number of such tiles which meet the patch $Tb_{L,N}^{\vep}(\mcJ_{L,N}^{\vep})$.  

%The idea of \cite{spain} is to partition the unit tiling of $\RR^{N+1}$ into those tiles which are translates of $[0,1)^{N+1}$ by an integer vector with \emph{even} components.  There are clearly $2^{N+1}$ disjoint families of ``every other tile'' which cover the plane.  If $\mcE$ denotes one such family, it is clear that the image under  $Tb_{L,N}^{\vep}$ of a cube of side length $1/\lceil KT\rceil$ meets at most 1 tile from $\mcE$.  Summing over the $2^{N+1}$ families, we find that $R(N,T)$ is at most $2^{N+1}$ times $Q_{L,N}^{\vep}(T)$, by which we denote the number of $N$-cubes of side length $1/\lceil KT\rceil$ required to cover $\mcJ_{L,N}^{\vep}.$  
Let $Q_{k,d}^{\vep}(T)$ denote the number of $d$-cubes of side length $1/\lceil KT\rceil$ required to cover $\mcJ_{k,d}^{\vep}.$  The easiest way to get an estimate for this quantity would be to note that each $\mcJ$ is contained in a cube of side length $2 \cdot {d \choose \lfloor d/2\rfloor}$.  However, we can do significantly better than this without too much effort, using 
the bounds on the individual coordinates (coefficients) from Lemma \ref{coefbound}.
 
Using (\refeq{coeffbound}), we see that $\mcJ_{k,d}^{\vep}$ is contained in the cuboid
\begin{align}
\left\{(x_1,\dots,x_k,y_0,\dots,y_{d-k-1}) \in \RR^{k}\times \RR^{d-k} ~\big|~|x_\ell| \leq {k\choose\ell},~|y_m| \leq {d-k\choose m},~\forall \ell, m\right\},
\end{align}
and therefore $\mcJ_{k,d}^{\vep}$ can be covered by 
\begin{equation} 
\prod_{\ell=1}^k 2{k \choose \ell} \cdot \prod_{m=0}^{d-k-1} 2{d-k \choose m} = 2^d P(k) \cdot P(d-k)
\end{equation}
unit $d$-cubes.  Hence surely we have
\begin{equation}\label{dumb}
Q_{k,d}^{\vep}(T) \leq 2^d P(k)P(d-k)\lceil KT\rceil^d \leq 2^d P(k)P(d-k)((K+1)T)^d.
\end{equation}
Using Theorem \ref{ourspain} we conclude that 
\begin{align}
|\mcM(\ld d,T) - \vol(\mcU_d)T^{d+1}| &\leq 
%R(d,T) \leq \sum_{L,\vep}R_{L,d}^{\vep}(T) \leq
2^{d+1}\sum_{k,\vep}Q_{k,d}^{\vep}(T)\\
&\leq 2^{d+1}\cdot 2 \sum_{k=0}^{d} 2^dP(k)P(d-k)(K+1)^dT^d\\
& = 4^{d+1} A(d)(K+1)^d T^d = \kappa_0(d)T^d.
\end{align}
We now estimate $\kappa_0(d)$ as in the statement of the theorem, using Lemma \ref{Aest} from the appendix:
\begin{align}
\kappa_0(d) &= 4^{d+1}A(d)\left(d{d \choose \lfloor d/2 \rfloor}+1\right)^d
\leq 4^{d+1}A(d)\left(2d{d \choose \lfloor d/2 \rfloor}\right)^d\\
&\leq 4^{d+1}A(d)\left(\frac{2e}{\pi}\sqrt{d}2^d\right)^d
\leq \left(40\sqrt[4]{2} \pi^{3/4}e^{-3}\right)d^{-1/4}\left(4\sqrt{2}e^{3/2}\pi^{-3/2}\right)^d\left(2\sqrt{e}\right)^{d^2}\\
%&\leq 40\sqrt[2]{2}\pi^{3/4} e^{-3} \cdot 4^{d}\sqrt{\frac{e^{d^2+d}}{(2\pi)^d d!}} \left(d{d \choose \lfloor d/2 \rfloor}+1\right)^d\label{2}\\
%& \leq 10 \sqrt{2} \pi e^{-3} \cdot \frac{4^{d+1}e^{d^2}}{(2\pi d)^{\frac{1}{2}d+\frac{1}{4}}} \left(\frac{e}{\pi}\sqrt{d}2^d+1\right)^d\label{3}\\ 
%&\leq 40 \sqrt[4]{2}\pi^{3/4}e^{-3} \cdot d^{-1/4}\cdot(4\sqrt{2}\pi^{-3/2}e)^d \cdot (2e)^{d^2}\label{4}\\
&= a \frac{b^d c^{d^2}}{\sqrt[4]{d}}\leq a (bc)^{d^2}
= 40 \sqrt[4]{2}\pi^{3/4}e^{-3} \cdot (8\sqrt{2}\pi^{-3/2}e^2)^{d^2}
\leq5.59 \cdot (15.01)^{d^2}\label{8},
%&\leq 6\cdot(16)^{d^2},\label{9} 
\end{align} 
where $a = 40 \sqrt[4]{2}\pi^{3/4}e^{-3},$ $b= 4\sqrt{2}e^{3/2}\pi^{-3/2},$ and $c = 2\sqrt{e}.$
% where $a = 40 \sqrt[4]{2}\pi^{3/4}e^{-3} \approx 5.59$, $b = 4\sqrt{2}e^{3/2}\pi^{-3/2} \approx 4.55$, $c = 2\sqrt{e} \approx 3.30,$ and $c_3 = c_1c_2 \approx 15.01$.

\end{proof}

\begin{remark}\normalfont
As each $\mcJ_{k,d}^{\vep}$ is measurable, it follows that for each $d$ we have 
\begin{equation}\label{qas}
Q_{k,d}^{\vep}(T) \sim \vol(\mcJ_{k,d}^{\vep}) \cdot ((K+1)T)^d,~ \textup{as}~T \to \infty.
\end{equation}
Notice that
\begin{equation}
\vol(\mcJ_{k,d}^{\vep}) = p_k(1) \cdot p_{d-k}(1),
\end{equation}
where $p_d(T)$ is as defined in (\refeq{fd}).  The sharpest way to proceed would be to explicitly estimate the error in (\refeq{qas}).  Comparing (\refeq{qas}) with (\refeq{dumb}): how much does $\vol(\mcJ_{k,d}^{\vep})$ differ from $2^d P(k)P(d-k)$?
\end{remark}

%%%%
%
%As in the proof of \cite[Theorem]{spain},\footnote{In this very short paper, the boundary is taken to be covered by the images of unit cubes -- we have adjusted accordingly.} the number of such unit tiles which meet the image $Tb_{L,N}^{\vep}\big(\mcJ_{L,N}^{\vep}\big)$
%is at most $2^{N+1}$ times the number of $N$-cubes of side length $1/\lceil KT\rceil$ needed to cover the set $\mcJ_{L,N}^{\vep}$.  By (\refeq{boxbound}), at most $\left(2 \cdot {N \choose \lfloor N/2 \rfloor} \cdot \lceil KT\rceil\right)^N$ such tiles are needed.  Since the boundary is covered by $2N+2$ such images, we conclude that
%\begin{align}
%R(N,T) &\leq (2N+2) \cdot 2^{N+1} \cdot \left(2 \cdot {N \choose \lfloor N/2 \rfloor} \cdot \lceil KT\rceil\right)^N\\
%&\leq (N+1) \cdot 2^{2N+2} \cdot {N \choose \lfloor N/2 \rfloor}^{2N} \cdot N^{N/2} \cdot \lceil T\rceil^N.
%\end{align}
%Since $\vol(T\mcU_N) = T^{N+1} \cdot \vol(\mcU_N)$, the theorem follows.
%\end{proof}

%8
\section{Counting monic polynomials: explicit bounds}\label{monicsec}
Let $\mcW_{d,T}$ denote the subset of $\RR^d$ corresponding to monic polynomials of degree $d$ in $\RR[z]$ with Mahler measure at most $T$, i.e.
\begin{equation}
\mcW_{d,T} = \{\vw = (w_1,\dots,w_d) \in \RR^d ~\big|~ \mu(1,\vw) \leq T\}.
\end{equation}
We want to estimate the number of lattice points $\mcM_1(d,T)$ in this region.  Note that, in the notation of the introduction, we have $\mcM_1(d,T) = \mcM(d,(1),(),T)$.  Recall that the volume of $\mcW_{d,T}$ is given by the Chern-Vaaler polynomial $p_d(T)$, as defined in (\refeq{fd}).

We define, for $d$ a non-negative integer, 
\begin{equation}
B(d) = \sum_{k=0}^d P(k)P(d-k)\gamma(k)^{d-k-1}\gamma(d-k)^k,
\end{equation}
where $P$ is as defined in (\refeq{P}), and $\gamma(k) := {k \choose \lfloor k/2 \rfloor}.$

\begin{theorem} \label{moniccount}
For all $d \geq 2$ and $T \geq 1$
%\begin{equation}\label{Tbig}
%T \geq {d-1\choose \left\lfloor\frac{d-1}{2}\right\rfloor},
%\end{equation}
 we have
\begin{equation}
\left|\mcM_1(d,T) - p_d(T)\right| \leq \kappa_1(d) T^{d-1},
\end{equation}
where
\begin{align}
\kappa_1(d) &= 4^d d^{d-1} B(d) \leq 4^d d^{d-1}2^{d^2}.
%&\leq 2^{3/4}\pi^{-1/4}\left(\frac{e}{\sqrt{2\pi}}\right)^d e^{d^2/2} \cdot d^{\frac{1}{2}d-\frac{5}{4}}\\
%&\leq 1.27 \cdot (2.95)^{d^2}.
\end{align}
\end{theorem}

\begin{proof}
Our starting point is the parametrization of the boundary $\partial W_{d,T}$ given in Section \ref{starsec}, which consists of the patches described in (\refeq{mpar1}) and (\refeq{mpar2}).  As opposed to the previous proof, we'll need to be a bit more careful in our application of Theorem \ref{ourspain}.  Instead of a Lipschitz estimate of the form $$\|\textup{output}_1-\textup{output}_2\|_\infty \leq [\textup{constant}]\cdot \|\textup{input}_1-\textup{input}_2\|_\infty,$$
we'll estimate each component of the parametrization separately, which will lead to an argument where the parameter space is tiled by ``rectangles'' instead of ``squares.''  We fix $k \in \{0,\dots,d-1\}$ and $\vep \in \{\pm1\}$, and set $\mcL =\mcL_{k,d}^{\vep T}$. %, and set $\mcL = \mcL_{k,d}^{\vep}.$  
We write 
\begin{equation}
\beta_{k,d}^{\vep T}(\vx,\vy) = \left(1,g_1(\vx,\vy),\dots,g_{d}(\vx,\vy)\right). 
\end{equation}
We have
\begin{align}
|g_i(\vx_1,\vy_1) - g_i(\vx_2,\vy_2)| \leq \sup_{(\vx,\vy) \in \mcL} \left|\nabla g_i(\vx,\vy) \cdot \left((\vx_1,\vy_1)-(\vx_2,\vy_2)\right)\right|\\
\leq \sup_{(\vx,\vy) \in \mcL} \left(\sum_{\ell=1}^k
\left|\frac{\partial g_i}{\partial x_\ell}{(\vx,\vy)}\right| 
|x_{1,\ell}-x_{2,\ell}|
 + 
 \sum_{m=1}^{d-k-1} 
\left|\frac{\partial g_i}{\partial y_m}{(\vx,\vy)}\right|  |y_{1,m}-y_{2,m}| \right).
\end{align}
By (\refeq{coeffbound}), if $(\vx,\vy) \in \mcL$, then we must have $|x_\ell| \leq {k \choose \ell} \leq \gamma(k)$, for each $\ell = 1,\dots,k$, and $|y_m| \leq T{d-k \choose m}$, for each $m = 1,\dots,d-k-1.$
Now notice that each partial derivative $\frac{\partial g_i}{\partial x_\ell}$, as a function, is either equal to $1$, $\vep T$, or $y_{i-\ell}$, and thus has absolute value at most $T{d-k \choose i-\ell} \leq T\gamma(d-k)$. %, using our hypothesis [NO WE DIDN'T!!!](\refeq{Tbig}). 
By the same token, each $\frac{\partial g_i}{\partial y_m}$ is equal to either $1$ or $x_{i-m}$, and thus has absolute value at most ${k \choose i-m} \leq \gamma(k)$.  Applying this to the inequality above gives
\begin{equation}\label{gd}
|g_i(\vx_1,\vy_1) - g_i(\vx_2,\vy_2)| \leq k\gamma(d-k)T\|\vx_1-\vx_2\|_\infty + (d-k-1)\gamma(k)\|\vy_1-\vy_2\|_\infty.
\end{equation}

%Recall that we write $\gamma(k) := {k \choose\lfloor k/2\rfloor}$ to simplify notation.  
%By our hypothesis (\refeq{Tbig}) we always have
%\begin{equation}\label{simpmax}
%\max\left\{T,{d-k\choose \left\lfloor\frac{d-k}{2}\right\rfloor}\right\} = T.  
%\end{equation}
Suppose for the moment that $0 < k < d-1$.  Now if $\frac{1}{p} + \frac{1}{q} = 1$, and if
\begin{align}
\|\vx_1-\vx_2\|_\infty &\leq \frac{1}{pk\gamma(d-k)T},~\textup{and}\\
\|\vy_1-\vy_2\|_\infty &\leq \frac{1}{q(d-k-1)\gamma(k)},
\end{align}
then (\refeq{gd}) will give
\begin{equation}
|g_i(\vx_1,\vy_1) - g_i(\vx_2,\vy_2)| \leq 1.
\end{equation}
So, if $\mcP$ is a cube in $\RR^k$ with sides parallel to the axes and side length 
\begin{equation}\label{Ps}
\frac{1}{\left\lceil p\gamma(d-k)kT\right\rceil},
\end{equation}
and if $\mcQ$ is a cube in $\RR^{d-k-1}$ with sides parallel to the axes and side length 
\begin{equation}\label{Qs}
\frac{1}{\left\lceil q(d-k-1)\gamma(k)\right\rceil},
\end{equation}
then $\beta_{k,d}^{\vep T}(\mcP \times \mcQ)$ is contained in a unit $d$-cube with sides parallel to the axes in $\RR^d$.  If $k = 0$, we take $q = 1$ in (\refeq{Qs}), and $\beta_{k,d}^{\vep T}(\mcQ)$ is contained in a unit $d$-cube with sides parallel to the axes in $\RR^d$.  Similarly, if $k=d-1$, then we take $p=1$ in (\refeq{Ps}), and we have the same result for $\beta_{k,d}^{\vep T}(\mcP)$ 

This is the first part of preparing to apply Theorem \ref{ourspain}.  We let $R_{k,d}^{\vep}(T)$ denote the minimum number of such ``rectangles'' $\mcP \times \mcQ$ required to cover $\mcL$.  As we argued in the previous section for the sets $\mcJ_{k,d}^{\vep}$, we see that $\mcL$ can be covered by 
\begin{equation}
\prod_{\ell=1}^k 2{k\choose \ell} \cdot \prod_{m=1}^{d-k-1}2T{d-k\choose m} = 2^{d-1}P(k)P(d-k)\cdot T^{d-k-1}
\end{equation}
unit cubes.  Since each unit cube can be covered by 
\begin{equation}
\left\lceil pk\gamma(d-k)T\right\rceil^k \cdot \left\lceil q(d-k-1)\gamma(k)\right\rceil^{d-k-1}
\end{equation}
of our rectangles, we have
\begin{align}
R_{k,d}^{\vep}(T) \leq 2^{d-1}P(k)P(d-k)\left\lceil pkT\right\rceil^k \cdot \left\lceil q(d-k-1)\gamma(k)\right\rceil^{d-k-1}T^{d-k-1},
\end{align}
for $0 < k < d-1$.  Similarly, when $k=0$ we have
\begin{align}
R_{k,d}^{\vep}(T)\leq 2^{d-1}P(k)P(d-k) \cdot  \left[(d-k-1)\gamma(k)\right]^{d-k-1}T^{d-k-1},\\
%&= 2^{d-1} P(d) (d-1)^{d-1}T^{d-1},
\end{align}
and when $k=d-1$ we have
\begin{align}
R_{k,d}^{\vep}(T) \leq 2^{d-1}P(k)P(d-k)\left[k\gamma(d-k)T\right]^k T^{d-k-1}.\\
%&= 2^{d-1} P(d) (d-1)^(d-1)T^{d-1},
\end{align}

Following the proof in the previous section, by Theorem \ref{ourspain}, we have
\begin{align}
&|\mcM_1(d,T) - p_d(T)| \leq \sum_{k,\vep}2^d R_{k,d}^{\vep}(T)\\ &\leq 2^d \cdot
2\sum_{k=0}^{d-1}
2^{d-1}P(k)P(d-k)\left\lceil pk\gamma(d-k)T\right\rceil^k \cdot \left\lceil q(d-k-1)\gamma(k)\right\rceil^{d-k-1}T^{d-k-1}\\
&= 4^d \sum_{k=0}^{d-1}
P(k)P(d-k)\left\lceil pk\gamma(d-k)T\right\rceil^k \cdot \left\lceil q(d-k-1)\gamma(k)\right\rceil^{d-k-1}T^{d-k-1},\label{eee}
\end{align}
where we understand $\left\lceil pk\gamma(d-k)T\right\rceil^k =1$ when $k=0$, and $\left\lceil q(d-k-1)\gamma(k)\right\rceil^{d-k-1} = 1$ when $k=d-1$, and similarly below.

It will now be convenient to set $p = \frac{d-1}{k}$ and $q = \frac{d-1}{d-k-1}$.  Note that if $k=0$ we have $q=1$, and $p$ does not appear; similarly if $k=d-1$ we have $p=1$, and $q$ does not appear.  We conclude our proof, assuming $T \geq 1$:
\begin{align}
|\mcM_1&(d,T) - p_d(T)| \leq\\
&4^d \sum_{k=0}^{d-1}P(k)P(d-k) (pk+1)^k (q(d-k-1)+1)^{d-k-1}\gamma(k)^{d-k-1}\gamma(d-k)^kT^{d-1}\\
= ~&4^d \sum_{k=0}^{d-1}P(k)P(d-k)d^k d^{d-k-1} \gamma(k)^{d-k-1}\gamma(d-k)^kT^{d-1}\\
= ~&4^d d^{d-1}B(d)T^{d-1} = \kappa_1(d)T^{d-1}.
\end{align}

Finally, we note that $B(d) \leq 2^{d^2}$ by Lemma \ref{Best} from the appendix.
\end{proof}
%We now estimate $\kappa_1(d)$ using Lemma \ref{Best} from the appendix.
%
%[FIX BELOW]
%
%\begin{align}
%\kappa_1(d) &= 4^d d^{d-1} B(d) \leq 4^d d^{d-1} \left(2^{3/4}\pi^{-1/4}\right)\left(\frac{e}{\sqrt{2\pi}}\right)^d e^{d^2/2} \cdot d^{-\frac{1}{2}d-\frac{1}{4}} \\
%&= 2^{3/4}\pi^{-1/4}\left(\frac{e}{\sqrt{2\pi}}\right)^d e^{d^2/2} \cdot d^{\frac{1}{2}d-\frac{5}{4}} %= a' \cdot b'^d \cdot c'^{d^2} \cdot d^{\frac{1}{2}d-\frac{5}{4}}\\
%%&\leq a'(b'c'\sqrt{e})^{d^2}= 2^{3/4}\pi^{-1/4} \left(\frac{e^2}{\sqrt{2\pi}}\right)^{d^2} \leq 3.48 \cdot (8.93)^{d^2},
%\leq 2^{3/4}\pi^{-1/4} \left(\frac{e^2}{\sqrt{2\pi}}\right)^{d^2} \leq 1.27 \cdot (2.95)^{d^2}.
%\end{align}
%%where $a' = 4\pi^{3/4} e^{-1},$ $b = 4\sqrt{2}\pi^{-3/2}e,$ and $c = \sqrt{4}{2}\sqrt{e}$.
%\end{proof}

%\begin{remark}\normalfont
%If $T$ is smaller than the requirement (\refeq{Tbig}), one can easily obtain an explicit bound similar to the one stated in the theorem, where the dependence on $d$ is worse, by following our proof and avoiding the simplification in (\refeq{simpmax}).
%\end{remark}

%9
\section{Lattice points in slices: explicit bounds}\label{slicessec}
The goal of this section is to prove a version of the lattice point-counting result Theorem \ref{mainthm} with an explicit error term, albeit with worse power savings -- Theorem \ref{slicecount} stated below.  As a byproduct of the proof, we also obtain an explicit version of our volume estimate Theorem \ref{volest}.  Our explicit version of Theorem \ref{mainthm} makes it possible to estimate the quantities in Corollaries \ref{unitcor} through \ref{normtracecor} with explicit error terms.

We start with some notation.  Fix $d, m,n, \vl,\vr,$ and $T>0$ as in Section \ref{introductionsec}, and again set $g=d-m-n$.  Let $\pi:\RR^{d+1} \to \RR^{g+1}$ denote the projection forgetting the first $m$ and last $n$ coordinates, given by
\begin{equation}
\pi(w_0,\dots,w_d) = (w_m,\dots,w_{d-n}).
\end{equation}
Let $S(T)$ be as defined in (\refeq{ST}).  For $t \in [0,\infty),$ define $W_t$ as in (\refeq{Wt}), and set 
\begin{equation}
B_t := \pi(W_t \cap \mcU_d).
\end{equation}
By (\refeq{STis}) we have
\begin{align}\label{piSTis}
\pi\left(S(T)\right) = \pi\left(T\left(W_{1/T} \cap \mcU_d\right)\right) = T\pi\left(\left(W_{1/T} \cap \mcU_d\right)\right) = TB_{1/T}.
\end{align}
Also note that by (\refeq{fromtau}) we have
\begin{equation}\label{volB0}
\vol(B_0) = \vol_{g+1}\left(\mcU_d \cap W_0\right) = V_g.
\end{equation}

For subsets $A$ and $A'$ of a common set, we use the usual notation for a symmetric difference $A \triangle A' = (A\cup A') \setminus (A \cap A').$  Note that for $T > 0$ we have 
\begin{equation}
T(A \triangle A') = (TA) \triangle (TA'),
\end{equation}
for any two subsets $A$ and $A'$ of a common euclidean space.

The following lemma is the main tool of this section.  We postpone its proof until the end.
%%%DONUT LEMMA
\begin{lemma}\label{donut}
Let
\begin{align}
k_1 &= k_1(d,\vl,\vr) :=2^{d^2}d^d(m+n)\|(\vl,\vr)\|_\infty,~\textup{and}\\
\delta_T &:= (k_1/T)^{1/d}.
\end{align}
If $T \geq k_1$, then
\begin{align}
B_0 \triangle B_{1/T} &\subseteq \{\vx \in \RR^{g+1} ~\big|~ 1-\delta_T \leq \mu(\vx) \leq 1+\delta_T\}\label{dtin1}\\
&= \left[(1+\delta_T)\mcU_g\right] \setminus \left[(1-\delta_T)\mcU_g\right].
\end{align}
\end{lemma}
Using this result we take a brief detour to make the advertised explicit volume estimate.  Compare the following with Theorem \ref{volest}, in which we obtain a better power-savings in the error term, though in that theorem the error term is not made explicit.
%%%%EXPLICIT VOLUME
\begin{theorem}\label{exvol}
Let $S(T) = S_{\vl,\vr}(T)$.
If $T \geq k_1,$ then
\begin{equation}
\left|\vol_{g+1}\left(S(T)\right) -  V_gT^{g+1}\right| \leq cT^{g+1-1/d},
\end{equation}
where
\begin{equation}
c = c\big(d,\vl,\vr\big) = 2^{d+1}\big((m+n)\|(\vl,\vr)\|_\infty\big)^{1/d}\cdot d\cdot V_g).
\end{equation}
\end{theorem}
\begin{proof}
Using (\refeq{piSTis}) and (\refeq{volB0}) we have
\begin{align}
\left|\frac{\vol_{g+1}(S(T))}{T^{g+1}}-V_g\right| &= \left|\vol(B_{1/T})- \vol(B_0)\right| \leq\vol(B_0 \triangle B_{1/T}) &\\
&\leq \vol(\{\vx \in \RR^{g+1} ~\big|~ 1-\delta_T \leq \mu(\vx) \leq 1+\delta_T\}) &\text{(by Lemma \ref{donut})}\\
&= 2\delta_TV_g = \frac{c}{T^{1/d}} .&%\text{(by (\refeq{int1}))},
\end{align}
\end{proof}

In Section \ref{volsec} we estimated the volume of $S(T)$ in order to estimate the number of lattice points in that set.  Here, by contrast, we actually don't require a volume estimate; Lemma \ref{donut} allows us to directly estimate the number of lattice points in $S(T)$, which we have denoted $\mcM(d,\vl,\vr,T)$, as follows.
%%%EXPLICIT LATTICE POINTS
\begin{theorem} \label{slicecount}
Let $k_1 = k_1(d,\vl,\vr)$ be as in Lemma \ref{donut}. For all $T \geq k_1$, we have
\begin{equation}
|\mcM(d,\vl,\vr,T) - V_g\cdot T^{g+1}| \leq \kappa(d,\vl,\vr)(T^{g+1-1/d}),
\end{equation}
where
\begin{align}
\kappa(d,\vl,\vr) &= (g+1)2^{g+1}k_1^{1/d}V_g + (g2^gk_1^{1/d}+1)\kappa_0(g).
\end{align}
\end{theorem}
We note for later that $V_g \leq 2\cdot15^{g^2}$ for all $g \geq 0$, and so
\begin{align}
\kappa(d,\vl,\vr) &\leq (g+1)2^{g+1}k_1^{1/d}\left(V_g+\kappa_0(g)\right)\\
&= d(g+1)2^{d+g+1}(m+n)^{1/d}\|\vl,\vr\|_\infty\left(V_g+\kappa_0(g)\right)\label{kappaest1}\\
&\leq (2+a)d(g+1)2^{d+g+1}(m+n)^{1/d}\|\vl,\vr\|_\infty (bc)^{g^2},
\end{align}
where $a$, $b$, and $c$ are the constants appearing in the end of the proof of Theorem \ref{genpolycount} (note that $bc>15$).

\begin{proof}
We let $Z(\Omega)$ denote the number integer lattice points in a subset $\Omega$ of euclidean space.
Again applying (\refeq{piSTis}), we have
\begin{align}
\mcM(d,\vl,\vr,T) = Z(S(T)) = Z(\pi(S(T)) = Z(TB_{1/T}).
\end{align}
Also note that 
\begin{equation}
Z(TB_0) = \mcM(g,T),
\end{equation}
which we estimated in Section \ref{cpebsec}.  Therefore, using the triangle inequality and Theorem \ref{genpolycount}, we have
\begin{align}
\left|\mcM(d,\vl,\vr,T) - V_g\cdot T^{g+1}\right| & = \left|Z(TB_{1/T}) - V_g\cdot T^{g+1}\right|\\
& \leq \left|Z(TB_{1/T}) - Z(TB_0) \right| + \left|Z(TB_0) - V_g\cdot T^{g+1}\right|\\
&\leq \left|Z(TB_{1/T}) - Z(TB_0) \right| + \kappa_0(g)T^g,\label{twoparts}
\end{align}
Clearly
\begin{align}
\left|Z(TB_{1/T}) - Z(TB_0) \right| &\leq Z\left((TB_{1/T}) \triangle (TB_0) \right) = Z\left(T(B_{1/T}\triangle B_0)\right),
\end{align}
and by Lemma \ref{donut} we have 
\begin{equation}
T(B_{1/T}\triangle B_0) \subseteq \left[(T+T\delta_T)\mcU_g\right]\setminus \left[(T-T\delta_T)\mcU_g\right].
\end{equation}
Hence, applying Theorem \ref{genpolycount} a second time and using an elementary estimate from the mean value theorem, we find that
\begin{align}
\left|Z(TB_{1/T}) - Z(TB_0) \right| &\leq Z((T+T\delta_T)\mcU_g) - Z((T-T\delta_T)\mcU_g)\\
&\leq V_g \left[(T+T\delta_T)^{g+1} - (T-T\delta_T)^{g+1}\right]\\
&+\kappa_0(g)\left[(T+T\delta_T)^g - (T-T\delta_T)^g\right]\\
&\leq V_g(g+1)(T+T\delta_T)^g(2T\delta_T) + \kappa_0(g) g (T+T\delta_T)^{g-1}(2T\delta_T).
\end{align}
 
 Recall that  $\delta_T = k_1^{1/d} T^{-1/d}.  $Assuming $T \geq k_1$ means that $\delta_T \leq 1$. Combining the estimate just obtained with (\refeq{twoparts}), we achieve
\begin{align}
|\mcM(d,\vl,\vr,T) - V_g\cdot T^{g+1}| &\leq V_g(g+1)(2T)^g\cdot 2T^{1-1/d} \cdot k_1^{1/d}\\
&+ g\kappa_0(g)(2T)^{g-1}\cdot 2T^{1-1/d}\cdot k_1^{1/d} + \kappa_0(g)T^g\\
&\leq [(g+1)2^{g+1}k_1^{1/d}V_g + (g2^gk_1^{1/d}+1)\kappa_0(g)]T^{g+1-\frac{1}{d}}.
\end{align}
\end{proof}

\begin{proof}[Proof of Lemma \ref{donut}]
We will require the following Lipschitz-type estimate for the Mahler measure \cite[Theorem 4]{chernvaaler}, which is a quantitative form of the continuity of Mahler measure: 

\begin{theorem}[Chern-Vaaler]
For any $\vw_1,\vw_2 \in \RR^{d+1},$ we have
\begin{equation}\label{cvlip}
\left| \mu(\vw_1)^{1/d} - \mu(\vw_2)^{1/d}\right| \leq 2\left\|\vw_1-\vw_2\right\|_1^{1/d},
\end{equation}
where $\|\vw\|_1 = \sum_{i=0}^{d} |w_i|$ is the usual $\ell^1$-norm of a vector $\vw = (w_0,\dots,w_d) \in \RR^{d+1}$.
\end{theorem}

If $\mu(\vw_1)$ and $\mu(\vw_2)$ are both less than some constant $k$, then applying (\refeq{cvlip}) yields
\begin{equation}\label{ourlip}
\left|\mu(\vw_1) - \mu(\vw_2)\right| = \left|\mu(\vw_1)^{1/d} - \mu(\vw_2)^{1/d}\right| \cdot \sum_{i=1}^{d} \left(\mu(\vw_1)^{\frac{d-i}{d}} \mu(\vw_2)^{\frac{i-1}{d}}\right) 
\leq 2\left\|\vw_1-\vw_2\right\|_1^{1/d} \cdot dk^{\frac{d-1}{d}}.
\end{equation}
We will shortly apply this observation with $k=2^d$.  We assume $T \geq k_1.$

Let $\vx$ be a vector in $B_0 \triangle B_{1/T}$, and write 
\begin{equation}
\vx_0 = \tau(\vx) = (\vec 0_m,\vx,\vec 0_n) \in \RR^{d+1},~~\textup{and}~~\vx_T = \left(\frac{\vl}{T},\vx,\frac{\vr}{T}\right) \in \RR^{d+1}.
\end{equation}
Notice that $\mu(\vx_0) =\mu(\vx)$ because $\tau$ preserves Mahler measure, as noted in the proof of Theorem \ref{volest}.

Since $\vx \in B_0 \triangle B_{1/T}$, it's clear that either
\begin{equation}\label{case1}
\mu(\vx_0) \leq 1 < \mu(\vx_T).
\end{equation}
or
\begin{equation}\label{case2}
\mu(\vx_T) \leq 1 < \mu(\vx_0).
\end{equation}
must hold. 
In either case, we have
\begin{equation}\label{hey}
1-|\mu(\vx_0) -\mu(\vx_T)| \leq \mu(\vx_0) \leq 1+ |\mu(\vx_0) -\mu(\vx_T)|
\end{equation}

First, suppose $\vx$ is in $B_0$, but not in $B_{1/T}$, so (\refeq{case1}) holds.  Then, by (\refeq{msup}) and our assumption that $T \geq k_1$, we have

\begin{equation}\label{yo}
%\mu(\vx_T) \leq \|\vx_T\|_\infty \sqrt{d+1} \leq \max\{ \|\vx_0\|_\infty,1\} \sqrt{d+1} \leq k \mu(\vx_0) \leq k,
\mu(\vx_T) \leq \|\vx_T\|_\infty \sqrt{d+1} \leq \max\{ \|\vx_0\|_\infty,1\} \sqrt{d+1} \leq {d\choose{\lfloor d/2\rfloor}}\sqrt{d+1} \max\{\mu(\vx_0),1\} \leq 2^d,
\end{equation}
as in the statement of the proposition.  Here we have used that ${d\choose{\lfloor d/2\rfloor}}\sqrt{d+1} \leq 2^d$ (see for example \cite[Lemma 1.6.12]{bombierigubler}).  Note that the second inequality in (\refeq{yo}) follows because $T \geq \|(\vl,\vr)\|_\infty$.  On the other hand, if  $\vx$ is in $B_0$, but not in $B_{1/T}$, so that (\refeq{case2}) holds, then by applying (\refeq{msup}) again, we have, in the same fashion as before:
\begin{equation}
%\mu(\vx_0) \leq \|\vx\|_\infty \sqrt{g+1} \leq \max\{\|\vx_T\|_\infty,1\}\sqrt{d+1} \leq k \mu(\vx_T) \leq k.
\mu(\vx_0) \leq \|\vx\|_\infty \sqrt{g+1} \leq \max\{\|\vx_T\|_\infty,1\}\sqrt{d+1} \leq  \max\{\mu(\vx_T),1\} \leq 2^d.
\end{equation}

Since in either case we have that both $\mu(\vx_0)$ and $\mu(\vx_T)$ are at most $2^d$, we may apply (\refeq{ourlip}) to achieve

\begin{equation}\label{hi}
|\mu(\vx_0)-\mu(\vx_T)| \leq 2\|\vx_0 - \vx_T\|_1^{1/d} \cdot d(2^d)^{\frac{d-1}{d}}.
\end{equation}
Note that
\begin{equation}
\|\vx_0 - \vx_T\|_1 = \sum_{i=0}^{m-1}|\ell_i|/T + \sum_{i=d-n+1}^{d}|r_i|/T \leq (m+n)\|(\vl,\vr)\|_\infty/T,
\end{equation}
which, combined with (\refeq{hi}), yields
\begin{equation}
|\mu(\vx_0)-\mu(\vx_T)| \leq \delta_T.
\end{equation}
Now we combine with (\refeq{hey}), and conclude that $1-\delta_T \leq \mu(\vx) \leq 1+\delta_T$.  This completes our justification of (\refeq{dtin1}), which concludes our proof of Lemma \ref{donut}.
\end{proof}
\section{Reducible and imprimitive polynomials}\label{sievingsec} 
In this section we begin to transfer our explicit counts for polynomials of degree at most $d$ to explicit counts for algebraic numbers of degree $d$, by counting their minimal polynomials.  In most cases, this simply means bounding the number of reducible polynomials, because the hypotheses imposed in Theorem \ref{maincor} don't allow for any irreducible polynomials to be counted other than minimal polynomials of degree $d$.  We'll apply a version of Hilbert's irreducibility theorem to achieve the most general bound, which will finish off the proof of Theorem \ref{maincor}.  However, in various special cases we work a little harder to improve the power savings.

In the one case we consider outside the hypotheses of Theorem \ref{maincor}, namely polynomials with no coefficients fixed, we must also address the presence of imprimitive degree $d$ polynomials and lower-degree polynomials.  

Several times in our arguments we use the following estimate: if $a \geq 2$, then
\begin{align}\label{thingy} 
\sum_{k=1}^K a^k = \frac{a^{K+1}-a}{a-1} \leq \frac{a^{K+1}}{a/2} = 2a^K.
\end{align}
We write
\begin{align}
P(d) &:= \prod_{j=0}^d {d \choose j}, ~\textup{for}~d \geq 0,~\textup{and}\\
C_{m,n}(d) &:= \prod_{j=m}^{d-n} \left(2{d \choose j} + 1\right),~\textup{for}~ 0 \leq m+n \leq d. \\ 
\end{align}

%%%%%%%%%%%
\subsection{All polynomials}\label{allpolysec}
Let $\mcM(d,T)$ denote the number of integer polynomials of degree \emph{exactly} $d$ and Mahler measure at most $T$, and let $\mcM^{red}(d,T)$ denote the number of such polynomials that are reducible.  Recall that $\mcM(\ld d, T)$ denotes the number of integer polynomials of degree \emph{at most} $d$ and Mahler measure at most $T$.  By (\refeq{coeffbound}), for all $d \geq 0$ and $T >0$ we have
\begin{equation}\label{upper}
\mcM(d,T) \leq \mcM(\ld d,T) \leq C_{0,0}(d)T^{d+1} \leq c_0 2^{d+1} P(d)T^{d+1},
\end{equation}
where $c_0 = 3159/1024$, using Lemma \ref{Mest} from the appendix.

\begin{proposition}\label{allred}
We have
\begin{align}
\mcM^{red}(d,T) \leq \left\{
\begin{array}{ll}
%18496\cdot T^2(\operatorname{log }T),~\textup{if} ~d = 2,~\textup{and}\\
1758\cdot T^2\log T, &\textup{if} ~d = 2,~T \geq 2,~\textup{and} \vspace{7pt}\\
16c_0^2 4^d P(d-1)\cdot T^d, &\textup{if}~d\geq 3,~T \geq 1.
\end{array}
\right.
\end{align}
\end{proposition}
\begin{proof}
For a reducible polynomial $f$ of degree $d$ and Mahler measure at most $T$, there exist $1 \leq d_2 \leq d_1 \leq d-1$ such that $f=f_1 f_2$, where each $f_i$ is an integer polynomial with deg$(f_i) = d_i$.  Of course we have $d=d_1 + d_2$.  Let $k$ be the unique integer such that $2^{k-1}\leq \mu(f_1) < 2^k$.  We have $1 \leq k \leq K$, where $K=\lfloor\frac{\text{log }T}{\text{log }2}\rfloor+1$, and $\mu(f_2) \leq 2^{1-k}T$.

Given such a pair $(d_1, d_2)$, by (\refeq{upper}) there are at most $c_0 2^{d_1+1} P(d_1)2^{k(d_1+1)}$ choices of such an $f_1$, and at most $c_0 2^{d_2+1}P(d_2)(2^{1-k}T)^{d_2+1}$ choices for $f_2$.  Assume first that $d_1>d_2$.  We'll use below that $P(d_1)P(d_2)$ is always less than or equal to $P(d-1),$ by Lemma \ref{Pesus} in the appendix.
Summing over all possible $k$ and applying (\refeq{thingy}), the number of pairs of polynomials is at most
\begin{align}
&\sum_{k=1}^K c_02^{d_1+1}P(d_1)c_0 2^{d_2+1}P(d_2)2^{k(d_1+1)} (2^{1-k}T)^{d_2+1}=4c_0^2 2^{d}P(d_1)P(d_2)(2T)^{d_2+1}\sum_{k=1}^K 2^{k(d_1-d_2)}\\
&\leq 4c_0^2 2^{d}P(d-1)(2T)^{d_2+1}\left[2 \cdot2^{K(d_1-d_2)}\right] \leq 8c_0^2 2^{d} P(d-1)(2T)^{d_1+1} \leq 16c_0^2 2^{d}2^{d_1}P(d-1)T^d.
\end{align} 
 
If instead $d_1=d_2=\frac{d}{2}$, (so in particular $d$ is even), then the first line above is at most 
\begin{equation}
4c_0^2 2^d P(d-1)(2T)^{d_1+1}K.
\end{equation}
In the case $d=2$, note that for $T \geq 2$ we have $K \leq \frac{2}{\log(2)} \log T$, and so
\begin{align}
\mcM^{red}(2,T) &\leq 4c_0^2 2^2 P(1)(2T)^{1+1}K \leq 64c_0^2T^2 \frac{2}{\log(2)}\log T \\&= \frac{128c_0^2}{\log(2)} \cdot T^2 \log T \leq 1758 \cdot T^2 \log T.
\end{align}
Whenever $T \geq 1$ we have $K \leq 2T$, and thus for even $d \geq 4$,
\begin{align}
4c_0^2 2^d P(d-1)(2T)^{d_1+1}K \leq 8c_0^2 2^d2^{d_1} P(d-1)T^{\frac{d}{2}+1}\cdot 2T\leq 16c_0^22^d2^{d_1}P(d-1)T^d,
\end{align}
so we have the same bound we had when we assumed $d_2 < d_1.$

Finally, for any $d \geq 3$, summing over the possible values of $d_1$ gives that
\begin{align}
\mcM^{red}(d,T) &\leq \sum_{d_1=\lceil \frac{d}{2} \rceil}^{d-1}16c_0^22^d2^{d_1}P(d-1)T^d \leq 16c_0^22^dP(d-1)T^d\sum_{d_1=1}^{d-1}2^{d_1} \\
&= 16c_0^22^dP(d-1)T^d (2^d-2)\leq 16c_0^2 4^d P(d-1)\cdot T^d.
\end{align}
\end{proof}

%%%%%%%PRIMITIVITY:
We follow the proof of \cite[Lemma 2]{masservaaler1} in counting primitive polynomials, but we'll keep track of implied constants.  For $n=1, 2, \dots$, let $\mcM^{n}(\ld d,T)$ denote the number of \emph{nonzero} integer polynomials of degree at most $d$ and Mahler measure at most $T$, such that the greatest common divisor of the coefficients is $n$.  We let $\mcM^{n}(d,T)$ denote the corresponding number of polynomials with degree \emph{exactly} $d$, so $\mcM^1(d,T)$ is the number of primitive polynomials of degree $d$ and Mahler measure at most $T$.  Recall that $\kappa_0(d)$ is a function of $d$ appearing in Theorem \ref{genpolycount}.
\begin{theorem}\label{allsieve}
For all $d \geq 2$ and $T \geq 1$ we have
\begin{align}
\left|\mcM^{1}(d,T) - \frac{V_d}{\zeta(d+1)}T^{d+1}\right| \leq 
\left(\frac{V_d}{d}+1\right)T  + \big(C_{0,0}(d-1) + \zeta(d)\kappa_0(d)\big)T^d,
\end{align}
where $\zeta$ is the Riemann zeta-function.
\end{theorem} 
\begin{proof}
Being careful to account for the zero polynomial, we have
\begin{equation}
\mcM(\ld d,T) -1= \sum_{1 \leq n \leq T}\mcM^{n}(\ld d,T) = \sum_{1 \leq n \leq T} \mcM^{1}\left(\ld d,T/n\right).
\end{equation}

By M\"{o}bius inversion (below we commit a sin of notation overloading and let $\mu$ denote the M\"{o}bius function), this tells us that
\begin{equation}
 \mcM^{1}(\ld d,T)=\sum_{1 \leq n \leq T}\mu(n)\left[\mcM\left(\ld d,T/n\right) - 1\right].
\end{equation}
%Similarly, we have
%\begin{equation}
% \mcM^{1}(d,T)=\sum_{1 \leq n \leq T}\mu(n)\mcM\left(d,T/n\right).
%\end{equation}
Combining this with Theorem \ref{genpolycount} and (\refeq{upper}), we have
%\begin{align}
%&\textup{bad version:}\\
%&\left|\mcM^{1}(d,T) - V_dT^{d+1}\sum_{1 \leq n \leq T}\frac{\mu(n)}{n^{d+1}}\right| \\
%&=\left|\mcM^{1}(d,T) - \sum_{n=1}^T \mu(n)\mcM(\ld d,T/n) +  \sum_{n=1}^T \mu(n)\mcM(\ld d,T/n) - V_dT^{d+1}\sum_{n=1}^T\frac{\mu(n)}{n^{d+1}}\right|\\
%&\leq \left|\sum_{n=1}^T \mu(n)\big[\mcM(d,T/n) - \mcM(\ld d,T/n)\big]\right| + \left|\sum_{n=1}^T \mu(n) \left[\mcM(\ld d,T/n) - V_d (T/n)^{d+1}\right]\right|\\
%&\leq \sum_{n=1}^T\mcM(\ld d-1,T/n) + \sum_{n=1}^T \left|\mcM(\ld d,T/n) - V_d (T/n)^{d+1}\right|\\
%&\leq \sum_{n=1}^T C_{0,0}(d-1) (T/n)^d + \sum_{n=1}^T \kappa_0(d)(T/n)^d \leq \zeta(d)\big(C_{0,0}(d-1) + \kappa_0(d)\big)T^d
%&\left|\mcM^{1}(d,T) - V_dT^{d+1}\sum_{1 \leq n \leq T}\frac{\mu(n)}{n^{d+1}}\right| \leq (\kappa_0(d)+C_{0,0}(d-1))T^d\sum_{1 \leq n \leq T} n^{-d}.
%\end{align}
\begin{align}
&\left|\mcM^{1}(d,T) - V_dT^{d+1}\sum_{1 \leq n \leq T}\frac{\mu(n)}{n^{d+1}}\right| \\
&=\left|\mcM^{1}(d,T) - \M^{1}(\ld d,T) +  \sum_{n=1}^T \mu(n)\big[\mcM(\ld d,T/n)-1\big] - V_dT^{d+1}\sum_{n=1}^T\frac{\mu(n)}{n^{d+1}}\right|\\
&\leq \mcM^1(\ld d-1,T) + \sum_{n=1}^T |\mu(n)| + \sum_{n=1}^T \left|\mcM(\ld d,T/n) - V_d (T/n)^{d+1}\right| \\ 
&\leq \mcM(\leq d-1,T) + T + \sum_{n=1}^T \kappa_0(d)(T/n)^d \leq C_{0,0}(d-1)T^d + T + \kappa_0(d)T^d \sum_{n=1}^T \frac{1}{n^d} \\
&\leq T + \big(C_{0,0}(d-1) + \zeta(d)\kappa_0(d)\big)T^d.
%&\leq \left|\sum_{n=1}^T \mu(n)\big[\mcM(d,T/n) - \mcM(\ld d,T/n)\big]\right| + \left|\sum_{n=1}^T \mu(n) \left[\mcM(\ld d,T/n) - V_d (T/n)^{d+1}\right]\right|\\
%&\leq \sum_{n=1}^T\mcM(\ld d-1,T/n) + \sum_{n=1}^T \left|\mcM(\ld d,T/n) - V_d (T/n)^{d+1}\right|\\
%&\leq \sum_{n=1}^T C_{0,0}(d-1) (T/n)^d + \sum_{n=1}^T \kappa_0(d)(T/n)^d \leq \zeta(d)\big(C_{0,0}(d-1) + \kappa_0(d)\big)T^d 
%&\left|\mcM^{1}(d,T) - V_dT^{d+1}\sum_{1 \leq n \leq T}\frac{\mu(n)}{n^{d+1}}\right| \leq (\kappa_0(d)+C_{0,0}(d-1))T^d\sum_{1 \leq n \leq T} n^{-d}.
\end{align}
This in turn gives 
\begin{align} 
\left|\mcM^{1}(d,T) - \frac{V_d}{\zeta(d+1)}T^{d+1}\right| &\leq V_dT^{d+1}\sum_{n=T+1}^\infty n^{-(d+1)} +  T + \big(C_{0,0}(d-1) + \zeta(d)\kappa_0(d)\big)T^d\\
&\leq \left(\frac{V_d}{d}+1\right)T  + \big(C_{0,0}(d-1) + \zeta(d)\kappa_0(d)\big)T^d,
\end{align}
by applying the integral estimate 
\begin{equation}
\sum_{n=T+1}^\infty n^{-(d+1)} \leq d^{-1}T^{-d}.
\end{equation}  This establishes the proposition.
\end{proof}

%%%%%%%%%%
\subsection{Monic polynomials}

Next, let $\mcM_1(d,T)$ denote the number of monic integer polynomials of degree $d$ and Mahler measure at most $T$, and let $\mcM_1^{red}(d,T)$ denote the number of such polynomials that are reducible.  Using (\ref{coeffbound}), we have for all $d \geq 0$ and $T>0$ that
\begin{equation}\label{zz}
\mcM_1(d,T) \leq C_{1,0}(d)T^d \leq c_1 2^d P(d) T^d,
\end{equation}
where $c_1 = \frac{1053}{512}$, from Lemma \ref{Mest} in the appendix.

%We'd like a constant $c_1(d)$ such that $\mcM_1(d,T) \leq c_1(d)T^d$.  By Theorem \ref{moniccount}, we know that $\mcM_1(d,T) \leq p_d(T) + \kappa_1(d)T^{d-1}$, for $T \geq 1$.
%
%Now, by (a corrected version of)\cite[Lemma 26]{chernvaaler}, we have that $$p_d'(T) \leq V_d2^{-s}(s!)^{-1}d^{s+1}T^{d-1}$$ for all $T >1$, where $s=\lfloor \frac{d-1}{2} \rfloor$.  By the Mean Value Theorem applied at $1$ and $T$, this gives that $$p_d(T) -p_d(1) \leq (T-1)V_d2^{-s}(s!)^{-1}d^{s+1}T^{d-1}.$$  Since 
%\begin{equation}
%|p_d(1)|=|V_d2^{-s}(s!)^{-1}\sum_{i=0}^s(-1)^i(d-2i)^s{s \choose i}| \leq |V_d(s!)^{-1}d^{s+1}|,
%\end{equation}
%we get that for $T \geq 1$ we can take 
%\begin{equation}
%c_1(d)= \left\{
%\begin{array}{ll}
%V_d2^{-s}(s!)^{-1}d^{s+1} + \kappa_1(d), & \textup{if}~d \leq 15,~\textup{and}\\
%V_{15}2^{-s}(s!)^{-1}d^{s+1} + \kappa_1(d), & \textup{if}~d > 15.
%\end{array}
%\right.
%\end{equation}
%As with $c_0$ in the previous subsection, we can easily see that $c_1(d)$ is an increasing function of $d$.

We'll assume $d \geq 2$.  In estimating the number of reducible monic polynomials, we follow the pattern of the proof of Proposition \ref{allred}, noting that if a \emph{monic} polynomial is reducible, its factors can be chosen to be monic.   Using the same notation as in that proof, we have that the number of pairs of monic polynomials of degree $d_1$ and $d_2$, with $d_1 > d_2$, is at most 
\begin{align}
\sum_{k=1}^K c_12^{d_1}P(d_1)c_1 2^{d_2}P(d_2)2^{kd_1} (2^{1-k}T)^{d_2} &=c_1^2 2^{d}P(d_1)P(d_2)(2T)^{d_2}\sum_{k=1}^K 2^{k(d_1-d_2)}\\
&\leq 2c_1^2 2^d 2^{d_1} P(d-1)T^{d-1}.
\end{align}
Noting that
\begin{equation}
\frac{16c_1^2}{\log 2} < 98,
\end{equation}
we continue almost exactly as in Proposition \ref{allred} and obtain the following.
\begin{proposition}\label{monicred}
We have
\begin{align}
\mcM_1^{red}(d,T) \leq \left\{
\begin{array}{ll}
98\cdot T\log T, &\textup{if} ~d = 2,~T \geq 2,~\textup{and} \vspace{7pt}\\
2c_1^2 4^d P(d-1)\cdot T^{d-1}, &\textup{if}~d\geq 3,~T \geq 1.
\end{array}
\right.
\end{align}
\end{proposition}

%%%%%%%%%%
\subsection{Monic polynomials with given final coefficient}

Next we want to bound the number of reducible, monic, integer polynomials with fixed constant coefficient.  For $r$ a nonzero integer, let $\mcM^{red}(d,(1),(r),T)$ denote the number of reducible monic polynomials with constant coefficient $r$, degree $d$, and Mahler measure at most $T$.  Using (\ref{coeffbound}), we have for all $d \geq 0$ and $T>0$ that
\begin{equation}\label{zz}
\mcM(d,(1),(r),T) \leq C_{1,1}(d)T^{d-1} \leq c_2 2^{d-1} P(d) T^{d-1},
\end{equation}
where $c_2 = \frac{351}{256}$, from Lemma \ref{Mest} in the appendix.

Let $\omega(r)$ denote the number of positive divisors of $r$.  We'll assume $d > 2$; if $d=2$, we easily have the constant bound $\mcM^{red}(d,(1),(r),T) \leq \omega(r)+1$.

For a polynomial $f$ counted by $\mcM^{red}(d,(1),(r),T)$, there exist $1 \leq d_2 \leq d_1 \leq d-1$ such that $f=f_1 f_2$, where each $f_i$ is an integer polynomial with deg$(f_i) = d_i$, and of course the constant coefficient of $f$ is the product of those of $f_1$ and $f_2$.  Define $k$ as in the previous two cases.  Given such a pair $(d_1,d_2)$, summing over the $2 \omega(r)$ possibilities for the final coefficient of $f_1$ there are at most $2\omega(r)c_22^{d_1-1}P(d_1)2^{k(d_1-1)}$ choices of such an $f_1$, and then at most $c_22^{d_2-1}P(d_2)(2^{1-k}T)^{d_2-1}$ choices for $f_2$.  The rest proceeds essentially as before, and we find that:
\begin{proposition}\label{normsieveprop}
For $T \geq 1$, we have

\begin{align}
\mcM^{red}(d,(1),(r),T) \leq \left\{
\begin{array}{ll}
\omega(r)+1, &\textup{if} ~d = 2\vspace{7pt}\\
\frac{1}{2} \omega(r)  c_2^2 4^d P(d-1)\cdot T^{d-2}, &\textup{if}~d\geq 3.
\end{array}
\right.
\end{align}
\end{proposition}

%%%%%%%%%
\subsection{Monic polynomials with a given second coefficient}
For our next  case, we want to bound the number of reducible, monic, integer polynomials with a given second leading coefficient.  Let $\mcM^{red}(d,(1,t),(),T)$ denote the number of reducible monic polynomials of degree $d\geq 3$ (we'll treat $d=2$ separately at the end) with integer coefficients, second leading coefficient equal to $t$, and Mahler measure at most $T$.

\begin{proposition}\label{tracesieveprop}
For all $t \in \ZZ$ we have
\begin{align}
\mcM^{red}(d,(1,t),(),T) \leq \left\{
\begin{array}{ll}
\frac{1}{2}\displaystyle{\sqrt{t^2+4T}+1}, &\textup{if}~d=2,~T \geq 1;\vspace{7pt}\\
\displaystyle{\frac{96}{\log 2}\cdot T\log T}, &\textup{if}~d=3,~T \geq 2;~\textup{and}\vspace{7pt}\\
\displaystyle{d 2^{2d-1} P(d-1) \cdot T^{d-2}}, &\textup{if}~d\geq 4,~T \geq 1.
\end{array}
\right.
\end{align}
\end{proposition}

\begin{proof}
%Let $$Q(d) = \prod_{i=2}^d \left[{d \choose i} +1\right], ~\textup{and}~Q(1)=1.$$
As before, we write such a polynomial as $f=f_1f_2$, with
\begin{equation}
f_1(z) = z^{d_1} + x_1z^{d_1-1} + \cdots x_{d_1},~\textup{and}~
f_2(z) = z^{d_2} + y_1z^{d_2-1} + \cdots y_{d_2}.
\end{equation}
Also as before, we enforce $1 \leq d_2 \leq d_1 \leq d-1$ to avoid double-counting, and we define $k$ as in the previous three cases.  %Assume first that $d_1 > d_2+1.$  Again we take $k$ between 1 and $K = \left\lfloor \frac{\log T}{\log 2}\right\rfloor+1,$ and assume that $2^{k-1} \leq \mu(f_1) \leq 2^k$, so $\mu(f_2) \leq 2^{1-k}T$.  
For $1 \leq i \leq d_1$ and $1 \leq j \leq d_2$, we have
\begin{equation}\label{cobds}
|x_i| \leq {d_1 \choose i} 2^k,~\textup{and}~|y_j| \leq {d_2 \choose j} 2^{1-k}T.
\end{equation}
We also, of course, have 
\begin{equation}\label{tracesum}
x_1 + y_1 = t.
\end{equation}
First assume $d_1 > d_2 + 1.$  Observe that the number of integer lattice points $(x_1,y_1)$ in $[-M_1,M_1] \times [-M_2,M_2]$ such that $x_1+y_1=t$ is at most $2 \min\{M_1,M_2\}+1$.  So the number of $(x_1,\dots,x_{d_1},y_1,\dots,y_{d_2})$ satisfying (\refeq{cobds}) and (\refeq{tracesum}) is at most
\begin{align}\label{tracebound}
&\left(2\min\{d_12^k, d_2 2^{1-k}T\}+1\right) \prod_{j=2}^{d_1}\left[2{d_1\choose j}2^k + 1\right] \cdot \prod_{j=2}^{d_2}\left[2{d_2 \choose j}2^{1-k}T+1\right]\\
&\leq\left(2\min\{d_12^k, d_2 2^{1-k}T\}+1\right)\cdot C_{2,0}(d_1) 2^{k(d_1-1)} \cdot C_{2,0}(d_2) (2^{1-k}T)^{d_2-1}\\
&\leq \left(2d\cdot 2^{1-k}T\right) (2T)^{d_2-1} 2^{k(d_1-d_2)} \cdot 2^{d_1-1} P(d_1) \cdot 2^{d_2-1} P(d_2)\\
&\leq d 2^{d-1} P(d-1)(2T)^{d_2} 2^{k(d_1-d_2-1)},
\end{align}
using Lemma \ref{Mest}.  Summing over all the possibilities $1 \leq k \leq K$, the number of possible pairs $f_1$ and $f_2$ of degrees $d_1$ and $d_2$, respectively, is at most 
\begin{align}
d2^{d-1}P(d-1)(2T)^{d_2} \sum_{k=1}^K 2^{(d_1-d_2-1)k} &\leq d 2^{d-1} 2^{d_2} P(d-1)T^{d_2} \left[2 \cdot 2^{K(d_1-d_2-1)}\right]\\
&\leq d2^{d-1} 2^{d_1} P(d-1)T^{d-2}.\label{botz}
\end{align}
Now, if $d_1 = d_2 = \frac{d}{2}$ (in this case $d$ must be even), then the geometric sum above becomes $\sum_{k=1}^K{2^{-k}} \leq 1$.  So for $d \geq 4$ again we obtain the estimate (\refeq{botz}) we achieved assuming $d_1 > d_2+1$.  If $d_1 = d_2 + 1$ (so $d$ is odd), then the number of possible pairs is at most
$d 2^{d-1} P(d-1) (2T)^{d_2} K,$ which does not exceed (\refeq{botz}) for $d \geq 5$, and for $d = 3$, $T \geq 2$ is at most
\begin{equation}
3 \cdot 2^{3-1} P(2) (2T)^1 \frac{2 \log T}{\log 2} = \frac{96}{\log 2} \cdot T \log T,
\end{equation}
which gives us the $d=3$ case of the proposition.  Finally, for $d \geq 4$ we sum over the at most $d/2$ possibilities for $(d_1,d_2)$, yielding
\begin{align}
\mcM^{red}(d,(1,t),(),T) &\leq d 2^{2d-1} P(d-1) T^{d-2}.
\end{align}

For the case $d=2$, we'll see that the error term is on the order of $\sqrt{T}$.  Note that we are simply counting integers $c$ such that the polynomial
\begin{equation}
f(z) = (z^2+tz+c) = (z+x_1)(z+y_1)
\end{equation}
has Mahler measure at most $T$.  Since we know $|c| \leq T$, it suffices to control the size of $\{x_1 \in \ZZ \ | \ |x_1(t-x_1)| \leq T\}$, which is itself bounded by the size of $\{x_1 \in \ZZ \ | \ x_1^2 - tx_1 \leq T\}$.  By the quadratic formula, that last set is simply $\{x_1 \in \ZZ \ | \ \frac{t-\sqrt{t^2+4T}}{2} \leq x_1 \leq \frac{t+\sqrt{t^2+4T}}{2}\}$, which has size at most $\sqrt{t^2+4T} + 1$.  To better bound the number of $c$ of the form $x_1(t-x_1)$, note that such a $c$ can be written in this form for exactly two values of $x_1$, except for at most one value of $c$ for which $x_1$ is unique (this occurs when $t$ is even).  So overall, the number of such $c$ with $|c| \leq T$ is at most $\frac{1}{2}\sqrt{t^2+4T}+1$.
\end{proof}

%%%%%%%%%%%
\subsection{Monic polynomials with given second and final coefficient}
For our final case, we want to bound the number of monic, reducible polynomials with a given second leading coefficient $t \in \ZZ$ and given constant coefficient $0 \neq r \in \ZZ$.  We can clearly assume that $d\geq 3$ since we're imposing three coefficient conditions.  We write $\mcM^{red}(d,(1,t),(r),T)$ for the number of reducible monic polynomials of degree $d$ with integer coefficients, second leading coefficient equal to $t$, and constant coefficient equal to $r$.  We'll show this is $O(T^{d-3})$ in all cases.  While we don't write an explicit bound for the error term, it should be clear from our proof that this is possible.
\begin{proposition}\label{ntsieveprop}
For all $d \geq 3$, $t \in \ZZ$, and $r \in \ZZ \setminus \{0\}$, we have
\begin{equation}
\mcM^{red}(d,(1,t),(r),T) =  O\left(T^{d-3}\right).
\end{equation}
\end{proposition}
\begin{proof}
As before, we write such a polynomial as $f=f_1f_2$, with
\begin{equation} 
f_1(z) = z^{d_1} + x_1z^{d_1-1} + \cdots x_{d_1},~\textup{and}~
f_2(z) = z^{d_2} + y_1z^{d_2-1} + \cdots y_{d_2}.
\end{equation}
We always enforce $1 \leq d_2 \leq d_1 \leq d-1$ to avoid double-counting.  We'll consider the count in several different cases.  First, if $d_2=1$, then $f_2 = z + y_{d_2}$, so we must have $y_{d_2} | r$ and $y_{d_2}+x_1 = t$.  Thus there are only $2\omega(r)$ possible choices of $f_2$; each choice will in turn determine $x_{d_1}$ and $x_1$, so we have $O(T^{d_1-2})=O(T^{d-3})$ choices of $f_1$ altogether, by Theorem \ref{mainthm}.  Note that this completely covers the case $d=3$.

Now assume $d_2 \geq 2$, so $d\geq 4$.  There are again only $2\omega(r)$ possible choices of $y_{d_2}$, and each one will determine what $x_{d_1}$ is (they must multiply to give $r$).  Fix a choice of $y_{d_2}$ for now.

Assume first that $d_1 > d_2+1.$  Again we take $k$ between 1 and $K = \left\lfloor \frac{\log T}{\log 2}\right\rfloor+1,$ and assume that $2^{k-1} \leq \mu(f_1) \leq 2^k$, so $\mu(f_2) \leq 2^{1-k}T$.  Almost exactly as in (\refeq{tracebound}), we get that the number of $(x_1,\dots,x_{d_1-1},y_1,\dots,y_{d_2-1})$ contributing to $\mcM^{red}(d,(1,t),(r),T)$ is at most
\begin{align}
&\left(2\min\{d_12^k, d_2 2^{1-k}T\}+1\right)\cdot\prod_{i=2}^{d_1-1}\left[2{d_1\choose i}2^k+1\right]\cdot \prod_{j=2}^{d_2-1}\left[2{d_2 \choose j}(2^{1-k}T)+1\right]\\
&\leq \left(2d\cdot 2^{1-k}T\right)\cdot 2^{k(d_1-2)} C_{2,1}(d_1) \cdot (2^{1-k}T)^{d_2-2} C_{2,1}(d_2)\\
&= d 2^{d_2}C_{2,1}(d_1)C_{2,1}(d_2) T^{d_2-1} 2^{(d_1-d_2-1)k}\\
&\leq \frac{1}{64}d 2^d2^{d_2}P(d-1) T^{d_2-1} 2^{(d_1-d_2-1)k},
\end{align}
using Lemmas \ref{Mest} and \ref{Pesus}.
Summing over all the possibilities $1 \leq k \leq K$, the number of possible pairs $f_1$ and $f_2$ of degrees $d_1$ and $d_2$, respectively, is at most 
\begin{align}
\frac{1}{64}d 2^d2^{d_2}P(d-1) T^{d_2-1} \sum_{k=1}^K2^{(d_1-d_2-1)k}\label{topsy2}
&\leq \frac{1}{32}d 2^d2^{d_1}P(d-1)T^{d_1-2}
\leq \frac{1}{32}d 2^d2^{d_1}P(d-1)T^{d-3},
%&~dQ(d_1)Q(d_2) 2^{d_1+2d_2-4}T^{d_2-1} \sum_{k=1}^K 2^{(d_1-d_2-1)k} \label{topsy2}\\
%\leq~& dQ(d_1)Q(d_2) 2^{d_1+2d_2-3}T^{d_2-1} \cdot (2T)^{d_1-d_2-1}= dQ(d_1)Q(d_2) 2^{2d_1+d_2-4}T^{d_1-2}\\
%\leq~& dQ(d_1)Q(d_2) 2^{2d_1+d_2-4}T^{d-4},\label{bottomsy2}
\end{align}
which is certainly $O(T^{d-3})$.

Next, if $d_1 = d_2 = \frac{d}{2}$ (in this case $d$ must be even), then the expression in (\refeq{topsy2}), which contains a partial geometric sum that's bounded by 1, is at most
\begin{equation}
\frac{1}{64}d 2^d2^{d_2}P(d-1) T^{\frac{d}{2}-1},
\end{equation}
which is certainly $O(T^{d-3})$ since $d\geq 4$.  Lastly, if $d_1=d_2+1$, (so $d \geq 5$), then $d_2 \leq d-3$, and (using $K \leq 2T$) the expression in (\refeq{topsy2}) is at most 
\begin{align}
\frac{1}{64}d 2^d2^{d_2}P(d-1)T^{d_2-1} K \leq \frac{1}{32}d 2^d2^{d_2}P(d-1)T^{d_2} \leq \frac{1}{32} d 2^d2^{d_2}P(d-1)T^{d-3},
\end{align}
which is $O(T^{d-3})$.
Finally, we sum over the $2\omega(r)$ possibilities for $y_{d_2}$ and the at most $d/2$ possibilities for $(d_1,d_2)$ and obtain overall that $\mcM^{red}(d,(1,t),(r),T) = O(T^{d-3}).$
\end{proof}

%\section{Proofs of Theorem \ref{maincor} and corollaries}\label{finalcountsec}
%\input{finalcount.tex}

%11
\section{Explicit results}\label{exthmssec}\label{exsec}
Let $N(\QQbar_d,\mcH)$ denote the number of algebraic numbers of degree $d$ over $\QQ$ and height at most $\mcH$.  We give an explicit version of Masser and Vaaler's main theorem of \cite{masservaaler1}, which follows from Theorem \ref{genpolycount}, our explicit version of \cite[Theorem 3]{chernvaaler}.

\begin{theorem}
For all $d \geq 2$ and $\mcH \geq 1$, we have
\begin{align}
\left| N(\QQbar_d,\mcH) - \frac{d V_d}{2\zeta(d+1)}\mcH^{d(d+1)} \right| \leq \left\{
\begin{array}{ll}
16690 \cdot \mcH^4\log \mcH, &\textup{if}~d =2~\textup{and}~\mcH \geq \sqrt{2}\vspace{7pt}\\
3.37 \cdot (15.01)^{d^2}\cdot \mcH^{d^2}, &\textup{if}~d \geq 3~\textup{and}~\mcH \geq 1.
\end{array}\right.
\end{align}
\end{theorem}

\begin{proof}
We combine Proposition \ref{allred} and Theorem \ref{allsieve} to estimate the number of irreducible, primitive (i.e. having relatively prime coefficients) polynomials of degree $d$ and Mahler measure at most $\mcH^d$, and relatively prime coefficients; we write $\mcM^{irr,~prim}(d,\mcH^d)$ for this number.  Each pair of such a polynomial and its opposite corresponds to $d$ algebraic numbers of degree $d$ and height at most $\mcH$ (the roots).  So we have $N(\QQbar_d,\mcH) = \frac{d}{2}\mcM^{irr,~prim}(d,\mcH^d)$, and
\begin{align}
&\left| N(\QQbar_d,\mcH) - \frac{d V_d}{2\zeta(d+1)}\mcH^{d(d+1)} \right|\\
 &\leq \left|\frac{d}{2}\mcM^{irr,~prim}(d,\mcH^d) - \frac{d}{2}\mcM^1(d,\mcH^d)\right| + \left|\frac{d}{2}\mcM^1(d,\mcH^d) - \frac{d V_d}{2\zeta(d+1)}\mcH^{d(d+1)}\right|\\
&\leq \frac{d}{2}\left(\mcM^{red}(d,\mcH^d) + \left|\mcM^1(d,T) - \frac{V_d}{\zeta(d+1)}\mcH^{d(d+1)}\right|\right),
\end{align}
and it follows from Proposition \ref{allred} and Theorem \ref{allsieve} that
\begin{align}
(d/2)^{-1}\left| N(\QQbar_d,\mcH) - \frac{d V_d}{2\zeta(d+1)}\mcH^{d(d+1)} \right| \leq
&\left(\frac{V_d}{d}+1\right) \mcH^d + \big(C_{0,0}(d-1)+\zeta(d)\kappa_0(d)\mcH^{d^2}\\
 &+ \left\{
\begin{array}{ll}
1758 \mcH^4 \log (\mcH^2), &\textup{if}~d = 2~\textup{and}~\mcH^2 \geq 2\vspace{7pt}\\
16c_0^2 4^d P(d-1) \mcH^{d^2}. &\textup{if}~d \geq 3~\textup{and} \mcH^2 \geq 1.
\end{array}
\right.
\end{align}
Here $\kappa_0(d)$ is the constant from Theorem \ref{genpolycount}, and $c_0 = 3159/1024.$
The $d=2$ case of our Theorem follows immediately, as 
\begin{equation}
\left(\frac{V_2}{2}+1\right) +  C_{0,0}(1) + \zeta(2)\kappa_0(2) + 2\cdot 1758 = \left(\frac{8}{2}+1\right) + 8000\zeta(2)+9+ 3516 < 16690.
\end{equation}
We now turn to $d \geq 3$, where we have
\begin{equation}
\left| N(\QQbar_d,\mcH) - \frac{d V_d}{2\zeta(d+1)}\mcH^{d(d+1)} \right| \leq k_0(d) \cdot \mcH^{d^2}
\end{equation}
with %(note below that $\zeta(d) \geq 1$)
\begin{align}
k_0(d) &= \frac{d}{2}\left(1 + V_d/d + \zeta(d)\kappa_0(d) + C_{0,0}(d-1) + 16c_0^2 4^d P(d-1)\right)\\
&= \left[\zeta(d) + \frac{1}{\kappa_0(d)} + \frac{V_d}{d\kappa_0(d)} + \frac{C_{0,0}(d-1)}{\kappa_0(d)} + \frac{16c_0^2 4^d P(d-1)}{\kappa_0(d)}\right]\frac{d\kappa_0(d)}{2}.
\end{align}
Note that the quantity in brackets above decreases for $d \geq 3$ (for this it may be helpful to consult Lemma \ref{volmax} and compute a few values of $V_d$) and so is no more than
\begin{align}
\lambda_0 := \zeta(3) + \frac{1}{\kappa_0(3)} + \frac{V_3}{3\kappa_0(3)} + \frac{C_{0,0}(2)}{\kappa_0(3)} + \frac{16c_0^2 4^3 P(2)}{\kappa_0(3)}.
\end{align}
So, using the notation of the end of the proof of Theorem \ref{genpolycount}, we have
\begin{align}
\left| N(\QQbar_d,\mcH) - \frac{d V_d}{2\zeta(d+1)}\mcH^{d(d+1)} \right| &\leq k_0(d) \cdot \mcH^{d^2} \leq \lambda_0\frac{d\kappa_0(d)}{2}\cdot \mcH^{d^2} \leq \frac{\lambda_0}{2}a d^{3/4} b^d c^{d^2} \cdot \mcH^{d^2}\\
&\leq  \frac{a\lambda_0}{2}(bc)^{d^2} \cdot \mcH^{d^2}
\leq 3.37 \cdot (15.01)^{d^2} \cdot \mcH^{d^2}.
\end{align}
\end{proof}

Next, we record an explicit version of \cite[Theorem 1.1]{barroero14} in the case $k=\QQ$, i.e. an explicit estimate for the number of algebraic integers of bounded height and given degree over $\QQ$.  This explicit estimate follows from our Theorem \ref{moniccount}, which improved the power savings of \cite[Theorem 6]{chernvaaler}.  We write $N(\mcO_d,\mcH)$ for the number of algebraic integers of degree $d$ over $\QQ$ and height at most $\mcH$.

\begin{theorem}
We have
\begin{align}
\left|N(\mcO_d,\mcH) - d\cdot p_d(\mcH^d) \right| \leq
\left\{
\begin{array}{ll}
584\cdot \mcH^2\log \mcH, &\textup{if}~d = 2~\textup{and}~\mcH \geq \sqrt{2}\vspace{7pt}\\
1.13 \cdot 4^d d^{d} 2^{d^2}\cdot \mcH^{d(d-1)}, &\textup{if}~d \geq 3~\textup{and}~\mcH \geq 1.
\end{array}
\right.
\end{align}
\end{theorem}

\begin{proof}
We follow the idea of the previous proof.  Now that we require polynomials to be monic, we never count two irreducible polynomials with the same set of roots, and so combining Theorem \ref{moniccount} and Proposition \ref{monicred} we obtain:
\begin{align}
d^{-1}\left|N(\mcO_d,\mcH) - d \cdot p_d(\mcH^d) \right| \leq \kappa_1(d) \mcH^{d(d-1)} + 
\left\{
\begin{array}{ll}
98\mcH^{2}\log (\mcH^2), &\textup{if}~d = 2, ~\mcH^2 \geq 2\vspace{7pt}\\
2c_1^2 4^d P(d-1)\mcH^{d(d-1)}, &\textup{if}~d \geq 3, ~\mcH^2 \geq 1,
\end{array}
\right.
\end{align}
where $c_1 = 1053/512$.  We immediately have the $d=2$ case of our theorem, as $\kappa_1(2) = 96$.  Assuming $d \geq 3$, we have
\begin{equation}
\left|N(\mcO_d,\mcH) - d\cdot p_d(\mcH^d) \right| \leq k_1(d) \cdot \mcH^{d(d-1)},
\end{equation}
where
\begin{align}
k_1(d) &= d\kappa_1(d) + 2c_1^2 d4^d P(d-1)\\
& = d\kappa_1(d) \left[1+ \frac{2c_1^2 4^d P(d-1)}{\kappa_1(d)}\right].
\end{align}
The quantity in brackets decreases for $d \geq 3$, and so is no more than
\begin{equation}
\lambda_1 := 1+ \frac{2c_1^2 4^3 P(2)}{\kappa_1(3)} \leq 1.13,
\end{equation}
and the result follows from the estimate for $\kappa_1(d)$ stated in Theorem \ref{moniccount}.
\end{proof}

%%%%EXPLICIT GIVEN NORM
We can also prove an explicit version of our Corollary \ref{normcor}, albeit with worse power savings.

\begin{theorem}\label{exnorm}
For each $d \geq 2$, $\nu$ a nonzero integer, and $\mcH \geq d \cdot 2^{d+1/d}|\nu|^{1/d}$, we have
\begin{equation}
\left| \mcN_{\Nm=\nu}(d,\mcH) - d V_{d-2}\cdot \mcH^{d(d-1)}\right| \leq \left\{
\begin{array}{ll}
\left(64\sqrt{2|\nu|}+8\right)\cdot\mcH + 2\omega(\nu) + 2, &\textup{if}~d =2\vspace{7pt}\\
0.0000063|\nu|\omega(\nu) \cdot d^3 4^d (15.01)^{d^2} \cdot \mcH^{d(d-1)-1}, &\textup{if}~d\geq 3,
\end{array}
\right.
\end{equation}
where $\omega(\nu)$ is the number of positive integer divisors of $\nu$.
\end{theorem}
\begin{proof}
Our proof proceeds very similarly to the last two.  Let $r = (-1)^d\nu.$  Using Theorem \ref{slicecount} and Proposition \ref{normsieveprop}, we have for all $\mcH \geq d \cdot 2^{d+1/d}|\nu|^{1/d}$:
\begin{align}
&d^{-1}\left|\mcN_{\Nm = \nu}(d,\mcH) - d \cdot V_{d-2}\cdot \mcH^{d(d-1)}\right| \\
&\leq \kappa\big(d,(1),(r)\big)\mcH^{d(d-1-1/d)} + \left\{
\begin{array}{ll}
\omega(r)+1, &\textup{if}~d = 2\vspace{7pt}\\
\frac{1}{2}\omega(r) c_2^2 4^d P(d-1) \cdot \mcH^{d(d-2)} &\textup{if}~d \geq 3,
\end{array}
\right.
\end{align}
where $\kappa\big(d,(1),(r)\big)$ is as defined in Theorem \ref{slicecount}, and $c_2 = 351/256$.  Consider the case $d=2$.  
%In this case, Theorem \ref{slicecount} can be sharpened quite a bit, since the number of quadratic polynomials which are monic, have a given final coefficient, and have Mahler measure at most $T$ is at most $2\cdot{2\choose 1}\mcH^2 + 1 \leq 5\mcH^2.$  So we actually have
By definition (stated in Theorem \ref{slicecount}) we have
\begin{equation}
\kappa\big(2,(1),(r)\big) = (0+1)2^{0+1}\left[2^4\cdot2^2(1+1)|r|\right]^{1/2}V_0 + \left(0+1\right)\kappa_0(0) = 32\sqrt{2|r|}+4,
\end{equation}
using $V_0=2$ and $\kappa_0(0) = 4$.  Therefore

\begin{equation}
\left|\mcN_{\Nm = \nu}(2,\mcH) - 2 \cdot V_0\cdot \mcH^{2}\right| \leq 2 \left((32\sqrt{2|r|}+4) \mcH + \omega(r) + 1\right) = \left(64\sqrt{2|r|}+8\right)\cdot\mcH + 2\omega(r) + 2.
\end{equation}

Now we assume $d \geq 3,$ and we have
\begin{align}
\left|\mcN_{\Nm = \nu}(d,\mcH) - d \cdot V_{d-2}\cdot \mcH^{d^2-d}\right| \leq k_2(d,r) \mcH^{d^2-d-1},
\end{align}
where, using (\refeq{kappaest1}) and letting $a$, $b$, and $c$ be as in the end of the proof of Theorem \ref{genpolycount}, we have
\begin{align}
k_2(d,r) &= d\left(\kappa\big(d,(1),(r)\big) + \frac{1}{2}\omega(r) c_2^2 4^d P(d-1)\right) \\
&\leq d \cdot (2+a)d(d-1)2^{2d-1+1/d} |r| (bc)^{(d-1)^2} + \frac{d}{2}\omega(r) c_2^2 4^d P(d-1)\\
&\leq d^32^{2d-1}|r|\omega(r) (bc)^{d^2}\left[\frac{(2+a)d(d-1)2^{1/d}}{(bc)^{2d-1}\omega(r)d^2} + \frac{c_2^2 P(d-1)}{d^2(bc)^{d^2}|r|}\right]\\
&\leq d^3 2^{2d-1}|r|\omega(r) (bc)^{d^2}\left[\frac{(2+a)2^{1/d}}{(bc)^{2d-1}} + \frac{c_2^2 P(d-1)}{d^2(bc)^{d^2}}\right].\\
\end{align}
old:
\begin{align}
k_2(d,r) &= d\left(\kappa\big(d,(1),(r)\big) + \frac{1}{2}\omega(r) c_2^2 4^d P(d-1)\right) \\
&\leq d \cdot (2+a)d(d-1)2^{2d-1+1/d} |r| (bc)^{(d-1)^2} + \frac{d}{2}\omega(r) c_2^2 4^d P(d-1)\\
&\leq d^32^{2d-1}|r|\omega(r) (bc)^{d^2}\left[\frac{(2+a)d(d-1)2^{1/d}}{bc\omega(r)d^2} + \frac{c_2^2 P(d-1)}{d^2(bc)^{(d-1)^2}|r|}\right]\\
&\leq d^3 2^{2d-1}|r|\omega(r) (bc)^{d^2}\left[\frac{(2+a)2^{1/d}}{bc} + \frac{c_2^2 P(d-1)}{d^2(bc)^{(d-1)^2}}\right].\\
\end{align}
As the quantity in brackets just above decreases for $d \geq 3$, it does not exceed
\begin{equation}
\frac{(2+a)2^{1/3}}{(bc)^5} + \frac{c_2^2 P(2)}{3^2(bc)^{9}} \leq 0.0000126,
\end{equation}
completing our proof.
\end{proof}

%%%%%%%%%%%%%%%%%%%%%%%%%%%%%
%
%To get an explicit estimate for $\mcN_{\Nm=\nu}(d,\mcH)$ (with worse power savings), we write out
%\begin{align} 
%&\left| \mcN_{\Nm=\nu}(d,\mcH) - d\cdot V_{d-2}\cdot \mcH^{d(d-1)}\right| \\
%&\quad= \left| d\mcM^{irr}(d,(1),(r),\mcH^d) - d\cdot V_{d-2}\cdot \mcH^{d(d-1)}\right|\\
%&\quad\leq \left| d\mcM^{irr}(d,(1),(r),\mcH^d) - d\mcM(d,(1),(r),\mcH^d)\right| + \left| d\mcM(d,(1),(r),\mcH^d) -  d\cdot V_{d-2}\cdot \mcH^{d(d-1)}\right|\\
%&\quad\leq d\cdot \mcM^{red}(d,(1),(r),\mcH^d) + d \cdot \kappa\big(d,(1),(r)\big)\mcH^{d(d-1-\frac{1}{d})} \\
%&\quad\leq d^2 2^{d-1}\omega(r)\beta_r(d)^2 \mcH^{d(d-2)} + d\cdot \kappa\big(d,(1),(r)\big)\mcH^{d(d-1-\frac{1}{d})}\\
%&\quad\leq k\mcH^{d^2-d-1},
%\end{align}
%where $k = d^2 2^{d-1}\omega(r)\beta_r(d)^2 + d\cdot \kappa\big(d,(1),(r)\big).$
%By (\refeq{kappaest1}), and letting $a,b,$ and $c$ be as in Section \ref{cpebsec} we have
%\begin{equation}
%\kappa\big(d,(1),(r)\big) \leq (2+a)d(d-1)2^{2d-1+1/d}|r|(bc)^{(d-2)^2},
%\end{equation}
%and so
%\begin{align}
%k &\leq \frac{1}{2}d^2 2^d \omega(r) \lambda^2 d^2(d-1)^2 2^{4d+2/d}r^2(bc)^{2(d-2)^2} + \frac{2+a}{2}d^2(d-1)2^{2+1/d}|r|(bc)^{(d-2)^2}\\
%&\leq \left[\frac{1}{2}\lambda^2+\frac{2+a}{2}\right]r^2\omega(r)d^4(d-1)^2 2^{5d+2/d}(bc)^{2(d-2)^2}\\
%&\leq (bc)^8 \left[\frac{1}{2}\lambda^2+\frac{2+a}{2}\right] r^2 \omega(r) d^4(d-1)^2\left(\frac{64}{(bc)^8}\right)^{d}(bc)^{2d^2}\\
%&\leq (bc)^8 \left[\frac{1}{2}\lambda^2+\frac{2+a}{2}\right] r^2 \omega(r)(bc)^{2d^2}.
%\end{align}
%We have shown the following.

%

We can immediately state the following explicit unit count, since counting units amounts to counting algebraic integers of norm $\pm 1$.
%\begin{corollary}
%In the notation of Corollary \ref{unitcor}, for all $d \geq 2$ and $\mcH \geq d \cdot 2^{d+1/d}$ we have
%\begin{equation}
%\left|\mcN_{\textup{units}}(\mcH,d) - 2d\cdot V_{d-2}\cdot \mcH^{d^2-d}\right| \leq 2a_1(a_2)^{d^2}\mcH^{d^2-d-1}.
%\end{equation}
%\end{corollary}
\begin{theorem}\label{exnorm}
For each $d \geq 2$ and $\mcH \geq d \cdot 2^{d+1/d}$, we have
\begin{equation}
\left| N(\mcO^*_d,\mcH) - 2d V_{d-2}\cdot \mcH^{d(d-1)}\right| \leq \left\{
\begin{array}{ll}
\left(128\sqrt{10}\right)\mcH + 8, &\textup{if}~d =2\vspace{7pt}\\
0.0000126\cdot d^3 4^d (15.01)^{d^2} \cdot \mcH^{d(d-1)-1}, &\textup{if}~d\geq 3.
\end{array}
\right.
\end{equation}
\end{theorem}

Finally, since Proposition \ref{tracesieveprop} gives an explicit bound, it is also possible to obtain an explicit estimate for $\mcN_{\Tr=\tau}(d,\mcH)$ similar to that of Theorem \ref{exnorm}; we leave this to the interested reader.

%APPENDIX
\section*{Appendix: combinatorial estimates}
\setcounter{equation}{0}
This appendix contains estimates for the combinatorial functions appearing in the some of the constants in this paper.
For any integer $d \geq 0$, define
\begin{align}
P(d) &:= \prod_{j=0}^d {d \choose j};\\
C_{m,n}(d) &:= \prod_{j=m}^{d-n} \left(2{d \choose j} + 1\right),~\textup{for}~0 \leq m+n \leq d;\\
A(d) &:= \sum_{k=0}^d P(k)P(d-k),  ~\textup{and}\\
B(d) &:= \sum_{k=0}^{d-1}P(k)P(d-k)\gamma(k)^{d-k-1}\gamma(d-k)^k.
\end{align}
where $\gamma(k) := {k \choose \lfloor k/2 \rfloor}.$

Stirling's inequality is the following estimate for factorials, which we will use several times:
\begin{equation}\tag{A.1}\label{stir}
\sqrt{2\pi} \cdot k^{k+\frac{1}{2}} e^{-k} \leq k! \leq e \cdot k^{k+\frac{1}{2}} e^{-k}  , ~\forall k \geq 1.
\end{equation}
Using this we can easily see that
\begin{equation}\tag{A.2}\label{gammaest}
\gamma(k) \leq \frac{e\cdot 2^k}{\pi\sqrt{k}}.
\end{equation}

\begin{customlem}{A.1}\label{Aest}
For all $d \geq 1$ we have
%\begin{equation}
%A(N) \leq \left(10\pi\sqrt{2}e^{-3}\right) \cdot \Phi(N),
%\end{equation}
\begin{equation}
A(d) \leq \left(10\sqrt[4]{2}\pi^{3/4}e^{-3}\right)e^{\frac{1}{2}d^{2}+d}(2\pi)^{-d/2}d^{-\frac{1}{2}d-\frac{1}{4}}.
\end{equation}
%\begin{equation}
%\Phi(N) := \sqrt{\frac{e^{N^2+N}}{(2\pi)^N N!}}.
%\end{equation}
\end{customlem}
\begin{proof}
We write 
\begin{equation}
\Phi(d) := \sqrt{\frac{e^{d^2+d}}{(2\pi)^d d!}}.
\end{equation}
Note that of course the first and last factor appearing in the product $P(d)$ are 1, so they may be omitted when convenient.  Also notice that
\begin{equation}\label{otherp}
P(d) = \prod_{k=1}^d \frac{k^k}{k!}.
\end{equation}

Using Stirling's inequality we have
\begin{align}\tag{A.3}\label{Pest}
P(d) = \prod_{j=1}^d \frac{j^j}{j!} \leq \prod_{j=1}^d \frac{e^j}{\sqrt{2\pi j}} = \frac{\exp\left(\frac{1}{2}(d^2+d)\right)}{\sqrt{2\pi}^d \sqrt{d!}} =  \sqrt{\frac{e^{d^2+d}}{(2\pi)^d d!}}.\\
\end{align}
We therefore have
\begin{equation}\tag{A.4}\label{pest}
P(d) \leq \Phi(d), ~\forall d \geq 0.
\end{equation}

Now, for all $d \geq 1$, we have
\begin{align}
A(d)  &= \sum_{k=0}^d P(k)P(d-k)  \leq \sum_{k=0}^d \Phi(k)\Phi(d-k) \\
&= \sum_{k=0}^d 
 \sqrt{\frac{e^{k^2+k}}{(2\pi)^k k!}} \cdot \sqrt{\frac{e^{(d-k)^2+d-k}}{(2\pi)^{d-k} (d-k)!}}\\ 
&= \Phi(d) \sum_{k=0}^d \sqrt{d \choose k}e^{k^2-dk} = \Phi(d) \left(2 + \sum_{k=1}^{d-1} \sqrt{d \choose k}e^{k^2-dk}\right).\tag{A.5}\label{twoplus}
\end{align}
Now, since $k^2-dk = -k(d-k) \leq -(d-1)$ when $1 \leq k \leq d-1$, we can easily estimate the sum
\begin{equation}\tag{A.6}\label{dum}
\sum_{k=1}^{d-1} \sqrt{d \choose k}e^{k^2-dk} \leq 2^d \cdot e^{1-d} = e \cdot (2/e)^d.
\end{equation}
The interested reader will easily verify that
\begin{equation}\tag{A.7}\label{APhi}
\frac{A(d)}{\Phi(d)} \leq \frac{A(2)}{\Phi(2)} = 10\pi\sqrt{2}e^{-3} \approx 2.21198
\end{equation} 
for $0 \leq d \leq 8$, and by (\refeq{twoplus}) and (\refeq{dum}), we can easily check that
\begin{equation}
\frac{A(d)}{\Phi(d)} \leq 2 + e \cdot (2/e)^d < 2.2
\end{equation}
for $d \geq 9$.

Finally, we estimate $\Phi(d)$ using Stirling's inequality again:
\begin{align}\tag{A.8}\label{Phiest}
\Phi(d)  \leq  \sqrt{\frac{e^{d^2+d}}{(2\pi)^d}\cdot \frac{e^d}{\sqrt{2\pi d}\cdot d^d}} = e^{\frac{1}{2}d^2+d}(2\pi d)^{-\frac{1}{2}d - \frac{1}{4}}.
\end{align}
Combining with (\refeq{APhi}) completes the proof.
\end{proof}

\begin{customlem}{A.2}\label{Best}
For all $d \geq 0$ we have
\begin{equation}
B(d) \leq 2^{d^2}.
\end{equation}

\begin{proof}
We can readily verify the inequality for $d \leq 3$, so we'll assume below that $d \geq 4$, and proceed by induction.  Suppose that $B(d-1) \leq 2^{(d-1)^2}$. Notice that 
\begin{equation}\tag{A.9}\label{Pdown}
P(d) = \frac{d^d}{d!}P(d-1),
\end{equation}
and also that $\gamma(d) \leq 2 \gamma(d-1)$ for all $d \geq 1.$  We also easily have $P(d) \leq e^{\frac{1}{2}d^2+d}$ from the previous proof.  Using these facts, we have
\begin{align}
B(d) &= P(d-1) + \sum_{k=0}^{d-2}P(k)P(d-k)\gamma(k)^{d-k-1}\gamma(d-k)^k\\
&\leq P(d-1) + \sum_{k=0}^{d-2}P(k)\frac{(d-k)^{d-k}}{(d-k)!} P(d-k-1)\gamma(k)^{d-k-2}\gamma(k) 2^k\gamma(d-k-1)^k\\
&\leq P(d-1) + \sum_{k=0}^{d-2} \left[\frac{e^{d-k}2^k}{\sqrt{2\pi(d-k)}}\gamma(k+1)\right]P(k)P(d-k-1)\gamma(k)^{d-k-2}\gamma(d-k-1)^k\\
&\leq P(d-1) + \sum_{k=0}^{d-2} \left[\frac{e^{d-k}2^k}{\sqrt{2\pi(d-k)}}\frac{e\cdot 2^{k+1}}{\pi\sqrt{k+1}}\right]P(k)P(d-k-1)\gamma(k)^{d-k-2}\gamma(d-k-1)^k\\
&\leq P(d-1) + \sum_{k=0}^{d-2} \left[\frac{e\sqrt{2}}{\pi^{3/2}}\cdot \frac{e^d \left(4/e\right)^k}{\sqrt{(d-k)(k+1)}}\right]P(k)P(d-k-1)\gamma(k)^{d-k-2}\gamma(d-k-1)^k.
\end{align}
We note that $(d-k)(k+1) \geq d$ for $0 \leq k \leq d-2$, and continue:  %Also note that our inductive hypothesis can be written as
%\begin{equation}
%B(d-1) \leq 2^{(d-1)^2} = \frac{2}{4^d} 2^{d^2}.
%\end{equation}
\begin{align}
B(d) &\leq P(d-1) +  \left[\frac{e\sqrt{2}}{\pi^{3/2}}\cdot \frac{e^d \left(4/e\right)^d}{\sqrt{d}} \right] \sum_{k=0}^{d-2}P(k)P(d-1-k)\gamma(k)^{d-1-k-1}\gamma(d-1-k)^k\\
&=P(d-1) +  \left[\frac{e\sqrt{2}}{\pi^{3/2}}\cdot \frac{4^d}{\sqrt{d}} \right] B(d-1) \leq P(d-1) + \left[\frac{e\sqrt{2}}{\pi^{3/2}}\cdot \frac{4^d}{\sqrt{d}} \right] 2^{(d-1)^2}\\
&=P(d-1) + \left[\frac{e\sqrt{2}}{\pi^{3/2}}\cdot \frac{4^d}{\sqrt{d}} \right] \frac{2}{4^d}2^{d^2} = P(d-1) + \left[\frac{e\cdot 2^{3/2}}{\pi^{3/2}\sqrt{d}} \right] 2^{d^2}\\
&= \left[\frac{P(d)d!}{d^d2^{d^2}} + \frac{e\cdot 2^{3/2}}{\pi^{3/2}\sqrt{d}}\right]2^{d^2} \leq \left[\frac{e^{\frac{1}{2}d^2+d}\cdot e\sqrt{d}}{e^d2^{d^2}} + \frac{e\cdot 2^{3/2}}{\pi^{3/2}\sqrt{d}}\right]2^{d^2}\\
&= \left[e\sqrt{d}\left(\frac{\sqrt{e}}{2}\right)^{d^2} + \frac{e\cdot 2^{3/2}}{\pi^{3/2}\sqrt{d}}\right]2^{d^2} \leq 2^{d^2}~\textup{for}~d \geq 4.
\end{align}
\end{proof}
\end{customlem}

\begin{customlem}{A.3}\label{Mest}
We have
\begin{align}
C_{0,0}(d) &\leq \frac{3159}{1024}\cdot 2^{d+1}P(d),~\forall~d \geq 0;\tag{A.10}\label{M00}\\
C_{1,0}(d) &\leq \frac{1053}{512}\cdot 2^d P(d),\label{M10}~\forall~d \geq 0;\\ 
C_{1,1}(d) &\leq \frac{351}{256} \cdot 2^{d-1}P(d),\label{M11}~\forall~d \geq 1;\\
C_{2,0}(d) &\leq 2^{d-1}P(d),\label{M20}~\forall~d \geq 1;~\textup{and}\\
C_{2,1}(d) &\leq \frac{1}{2} \cdot 2^{d-2}P(d),~\forall~d \geq 2.\label{M21}
\end{align}
\end{customlem}
\begin{proof}
We'll prove the bound for $C_{0,0}(d)$, and leave the other cases as exercises.  The inequality (\refeq{M00}) is easily verified for $d \leq 3$, and we have equality for $d = 4$.  If we set
\begin{equation}
R(d) := \frac{C_{0,0}(d)}{2^{d+1}P(d)}  = \prod_{j=0}^d \frac{2{d \choose j}+1}{2{d \choose j}},
\end{equation} 
then to establish (\refeq{M00}) it will suffice to show that 
\begin{equation}\label{ratty}
\frac{R(d+1)}{R(d)} \leq 1, ~\textup{for}~d \geq 4.
\end{equation}
We'll use the standard identity
\begin{equation}
{d+1 \choose j} = \frac{d+1}{d+1-j}{d\choose j}.
\end{equation}
We have
\begin{align}
\frac{R(d+1)}{R(d)} &= \frac{\prod_{j=0}^{d+1} \frac{2{d+1 \choose j}+1}{2{d+1 \choose j}}}{\prod_{j=0}^d \frac{2{d \choose j}+1}{2{d \choose j}}} = \frac{3}{2}\prod_{j=0}^d \frac{{d \choose j}}{{d+1\choose j}}\cdot \frac{2{d+1 \choose j}+1}{2{d \choose j}+1}\\
 &= \frac{3}{2}\prod_{j=0}^d \frac{d+1-j}{d+1}\cdot \frac{2\frac{d+1}{d+1-j}{d \choose j}+1}{2{d \choose j}+1}
 = \frac{3}{2}\prod_{j=0}^d\frac{2{d \choose j}+\frac{d+1-j}{d+1}}{2{d \choose j}+1}\\
  &= \frac{3}{2}\prod_{j=0}^d\left[1-\frac{j}{(d+1)\left(2{d \choose j}+1\right)}\right]\\
  &\leq \frac{3}{2}\prod_{j=d-2}^d\left[1-\frac{j}{(d+1)\left(2{d \choose j}+1\right)}\right] \\
  &=\frac{3}{2}\cdot \frac{4d^6 + 10d^5 + 6d^4 + 8d^3 + 20d^2 + 24d + 18}{6d^6 + 15d^5 + 12d^4 + 9d^3 + 15d^2 + 12d + 3}\\
  &= \frac{2d^6 + 5d^5 + 3d^4 + 4d^3 + 10d^2 + 12d + 9}{2d^6 + 5d^5 + 4d^4 + 3d^3 + 5d^2 + 4d + 1} \leq 1,~\textup{for}~d \geq 4.
\end{align}
\end{proof}

\begin{customlem}{A.4}\label{Pesus}
If $d \geq 2$ and $1 \leq k \leq d-1$, then
\begin{equation}
P(k)P(d-k) \leq P(d-1).
\end{equation}
\end{customlem}
% \begin{proof}
% We have equality if $k = d-1$.  It suffices to prove that, if $k\geq d-k+1$ (i.e. $k >d/2$), then 
% \begin{equation}
% P(k)P(d-k) > P(k-1)P(d-k+1). 
% \end{equation}
% Using (\refeq{Pdown}), and since $n \mapsto \frac{n^n}{n!}$ is increasing, we have for $k \geq d-k+1$:
% \begin{align}
% \frac{P(k)P(d-k)}{P(k-1)P(d-k+1)} &= \frac{k^k}{k!}\cdot \frac{(d-k+1)!}{(d-k+1)^{d-k+1}} >\frac{(d-k+1)^{d-k+1}}{(d-k+1)!} \cdot \frac{(d-k+1)!}{(d-k+1)^{d-k+1}} = 1.
% \end{align}
% \end{proof}
\begin{proof}
We have
\begin{align}
P(k)P(d-k) &= \prod_{j=0}^{k-1} {k \choose j} \prod_{i=0}^{d-k-1} {d-k \choose i} \leq \prod_{j=0}^{k-1} {d-1 \choose j} \prod_{i=0}^{d-k-1} {d-1 \choose i}
\\
&= \prod_{j=0}^{k-1} {d-1 \choose j} \prod_{i=0}^{d-k-1} {d-1 \choose d-1-i} =\prod_{j=0}^{k-1} {d-1 \choose j} \prod_{j=k}^{d-1} {d-1 \choose j} = P(d-1).
\end{align}
We have equality if and only if $k=1$ or $k=d-1$.
\end{proof}

%\begin{appendices}
%\section{Combinatorial estimates} \label{app}
%\end{appendices} 

\bibliography{units}%{}

\end{document}